\documentclass[a4paper, english, 11 pt]{amsart}
\usepackage{amsaddr} 
\usepackage[utf8]{inputenc}
\usepackage[T1]{fontenc}
\usepackage{textcomp}

\usepackage{amsfonts}
\usepackage{amssymb}
\usepackage{amsmath}
\usepackage{amsthm}
\usepackage{dsfont}

\usepackage{epstopdf}
\usepackage[dvips]{graphicx}
\usepackage{graphicx}
\usepackage{xcolor}
\usepackage{xspace}

\usepackage{relsize}
\usepackage[bottom]{footmisc}
\usepackage[small]{caption}
\everymath{\displaystyle}
\usepackage{enumerate}
\usepackage{multirow}
\usepackage{hyperref}

\usepackage[
backend=biber,
style=alphabetic
]{biblatex}
\newtheorem{thm}{Theorem}[section]
\newtheorem{prop}[thm]{Proposition}
\newtheorem{lem}[thm]{Lemma}
\newtheorem{cor}[thm]{Corollary}
\theoremstyle{definition}
\newtheorem{defi}[thm]{Definition}
\newtheorem{ass}[thm]{Assumptions}

\theoremstyle{remark}
\newtheorem{rem}[thm]{Remark}
\numberwithin{equation}{section}

\let\cal\mathcal
\let\rm\mathrm

\let\bf\mathbf

\newcommand{\N}{\mathbb{N}}

\newcommand{\R}{\mathbb{R}}
\newcommand{\C}{\mathbb{C}}
\newcommand{\K}{\mathbb{K}}

\newcommand{\Pp}{\mathbb{P}} %% Probability
\newcommand{\E}{\mathbb{E}} %% Expectation
\DeclareMathOperator{\Cov}{Cov} %% Covariance

\newcommand{\mmp}{\mathrm{MP}} %% Marčenko Pastur distribution
\newcommand{\Tr}{\mathrm{Tr}} %% Trace
\newcommand{\Sp}{\mathrm{Sp}} %% Spectrum

\newcommand{\verti}[1]{{\left\vert #1 \right\vert}}
\newcommand{\vertii}[1]{{\left\vert\kern-0.25ex\left\vert #1 
    \right\vert\kern-0.25ex\right\vert}}
\newcommand{\vertiii}[1]{{\left\vert\kern-0.25ex\left\vert\kern-0.25ex\left\vert #1 \right\vert\kern-0.25ex\right\vert\kern-0.25ex\right\vert}}
\newcommand{\vertif}[1]{{\left\vert\kern-0.25ex\left\vert #1 
    \right\vert\kern-0.25ex\right\vert}_F} 
\newcommand{\vertim}[1]{{\left\vert\kern-0.25ex\left\vert #1 
    \right\vert\kern-0.25ex\right\vert}_{\max}}
\newcommand{\vertih}[1]{{\left\vert\kern-0.25ex\left\vert #1 
    \right\vert\kern-0.25ex\right\vert}_{\cal H}}    
\newcommand{\pt}[1]{{\left( #1 \right)}}
\newcommand{\br}[1]{{\left[ #1 \right]}}
\newcommand{\lint}{[\![}
\newcommand{\rint}{]\!]}

\newcommand{\Sigmal}{{\Sigma_{\rm{lin}}} }
\newcommand{\Sigmaa}{{\Sigma_{\rm{approx}}} }
\newcommand{\Sigmat}{{\tilde \Sigma}}
\newcommand{\Sigmax}{ {\Sigma_X}}

\begin{document}

%% TITLE PAGE 
\title[Deterministic equivalent of the Conjugate Kernel]{Deterministic equivalent of the Conjugate Kernel matrix associated to Artificial Neural Networks}

\author{Clément Chouard}\thanks{\textbf{Acknowledgments:} This work received support from the University Research School EUR-MINT (reference number ANR-18-EURE-0023).%\newline The author would like to thank who ??
}
\address{Institut de Mathématiques de Toulouse; UMR5219\\ Université de Toulouse; CNRS; UPS, F-31062 Toulouse, France}
\email{clement.chouard@math.univ-toulouse.fr}

\maketitle

%% ABSTRACT
\begin{abstract} We study the Conjugate Kernel associated to a multi-layer linear-width feed-forward neural network with random weights, biases and data. We show that the empirical spectral distribution of the Conjugate Kernel converges to a deterministic limit. More precisely we obtain a deterministic equivalent for its Stieltjes transform and its resolvent, with quantitative bounds involving both the dimension and the spectral parameter. The limiting equivalent objects are described by iterating free convolution of measures and classical matrix operations involving the parameters of the model.
\end{abstract}

\newpage

\tableofcontents

%%
%%  To enter text is easy.  Just type it.  A blank line starts a new
%%  paragraph. 
%%
\newpage

\section{Introduction}

Artificial Neural Networks are computing systems inspired by a simplified model of interconnected neurons, storing and exchanging information inside a biological brain. Theorized in the late fifties \cite{rosenblatt1958perceptron}, this class of algorithms only established oneself as a cutting-edge technology in the two thousands, thanks to the ever-increasing computing power, and perhaps even more thanks to the accessibility of huge databases in the modern-age internet. Compared to the delicate craftsmanship of artificial intelligence specialists, the mathematical understanding of these networks remained at a very basic level until recently. The classical tools of random matrix theory in particular were helpless regarding the intrinsic non-linear connections between the artificial neurons.

This article is concerned about the Conjugate Kernel model, first introduced from the point of view of random matrices in \cite{PW17}. This model is a random matrix ensemble that corresponds to a multi-layer feed-forward neural network with weights initialized at random, in the asymptotic regime where the network width grows linearly with the sample size. In this setting, the Conjugate Kernel  matrix is simply the sample covariance matrix of the output vectors produced by the final layer of the network. It can be shown that the Conjugate Kernel governs the training and generalization properties of the underlying network (\cite{advani2020high}, \cite{yang2019fine}). Its spectral properties are thus of great theoretical and practical interest.

The limiting spectral distribution of the Conjugate Kernel  associated to a single-layer network was first characterized in \cite{PW17} under the form of a quartic polynomial equation satisfied by the Stieltjes transform of the limiting measure. This was proven rigorously in \cite{BP19} under more general assumptions, using advanced combinatorics. Some results on the extremal eigenvalues of the model were even obtained in \cite{benigni2022largest}. \cite{P19} first noticed a connection between these equations and the measures obtained from a free multiplicative convolution with some Marčenko-Pastur distributions. Let us also mention the work of \cite{piccolo2021analysis}, where a variant of model including a rank-one additive perturbation is examined, also by means of combinatorial techniques.

A first step towards a deeper understanding of the model using analytical techniques was done in \cite{LC19} in the case of a single-layer network with deterministic data.  \cite{FW20} later generalized this result to random data and multi-layer networks, and was the first to convincingly explain the appearance of free probability convolutions in the limiting spectral distribution of the Conjugate Kernel matrix. A similar analysis was done in \cite{wang2023deformed} in the regime where the network width is much larger than the sample size.

The aforementioned results may be classified as global laws, in the sense that they establish the existence of a limiting spectral distribution, a problem that is directly linked to the convergence of the Stieltjes transform via the classical inversion formulas. The next natural question that arises in random matrix theory is to look for a deterministic equivalent of the resolvent of the Conjugate Kernel  matrix in the sense of \cite{hachem2007deterministic}. This consists in finding a deterministic matrix that is asymptotically close to the expected resolvent of the matrix, and thus close to the random resolvent itself provided enough concentration in the random matrices. Establishing a  deterministic equivalent result allows for instance to approximate  linear statistics of the eigenvectors \cite{yang2020linear}, to examine the convergence of the eigenvectors empirical spectral distributions, and possibly to study the outliers of spiked models \cite{noiry2021spectral}. Regarding classical sample covariance random ensembles, this task was done a decade ago in the series of articles \cite{pillai2014universality},  \cite{bloemendal2014isotropic}, \cite{knowles2017anisotropic}. Such results were recently partially extended to models with a general dependence structure in \cite{LC21}, \cite{moi1}.

The most important contribution of this article is Theorem \ref{ANNMulDetEq}, that provides a deterministic equivalent of the Conjugate Kernel of a multi-layer neural network model. This theorem extends previously known results in various directions. First we include models with non-differentiable activation functions, as well as potential biases inside or outside the activation function. Secondly we give a quantitative estimate for the convergence of the Stieltjes transforms, which translates into a quantitative convergence of the measures in Kolmogorov distance. Finally and most importantly, we obtain local results, taking the form of a deterministic equivalent for the resolvent matrix of the Conjugate Kernel, quantitatively on both the dimension and the spectral parameter.

In the rest of the paper, we also establish intermediary results that may be interesting by themselves. In particular we show that the free convolution of measures with a Marčenko-Pastur distribution is regular with respect to the Stieltjes transform of these distributions (Theorem \ref{72}). We obtain a similar property for the deterministic equivalent matrices that appear in our results (Proposition \ref{ApproxbgGDet}). We also show how our framework applies to other models involving entry-wise operations on random matrices in Section \ref{practicalex}.

This paper relies mostly on analytical methods that study spectral functions of random matrices. We make great use of the theory of Stieltjes transforms and resolvent matrices, particularly its recent developments towards the deterministic equivalent of  general sample covariance matrices (\cite{LC21}, \cite{moi1}). We introduce a new notion of asymptotic equivalence for objects that depend both on a dimension and on a complex spectral parameter. Our definition is convenient to work with, whilst keeping enough precision to imply some quantitative results like the convergence of measures in Kolmogorov distance.

This analytical toolbox works in conjunction with concentration of measure principles. We follow the framework of \cite{LC20}, in particular we use the notion of Lipschitz concentration that is remarkably compatible with entry-wise operations on random matrices. We also use a linearization argument to study 
the covariance matrices of functions of weakly correlated Gaussian vectors. Using the theory of Hermite polynomials, we provide estimates for this approximation, not only on an entry-wise basis like it was done in \cite{FW20}, but also in spectral norm which is new to our knowledge.

After we released the first version of this manuscript, we came across the preprint \cite{schröder2023deterministic}, which we were not aware of.
A future comparison of the two different perspectives would be certainly interesting, as \cite{schröder2023deterministic} studies similar objects and seems to describe related phenomena to those considered here.

 \subsection{Overview of the article}
 In Section \ref{Sec2}, we remind the notations and basic properties of concentrated random vectors and matrices, following the framework of \cite{LC20}. We also introduce the new notion of asymptotic equivalence that will be key in the rest of the article. In Section \ref{Sec3}, we address the problem of the approximation the covariance matrices of functions of weakly correlated Gaussian vectors. In Section \ref{Sec4}, we remind the general deterministic equivalent results on which this article is based, and we study thoroughly the properties of the matrices appearing in these equivalents. In Section \ref{Sec5}, we study a first model of artificial neural network with a single layer and deterministic data. We also explain how our framework applies to other models involving entry-wise operations on random matrices (\ref{practicalex}). In Section \ref{Sec6}, we analyze a second model of artificial neural network, still consisting of a single layer but with random data instead. In Section \ref{Sec7}, we explain how to study multi-layer networks by induction on the model with one layer. In the Appendix \ref{Sec8}, we prove an independent result about the convergence the convergence in Kolmogorov distance of some empirical spectral measures.

\subsection{General notations and definitions}

The set of matrices with $d$ lines, $n$ columns, and entries belonging to a set $\K$ is denoted as $\K^{d \times n}$. We use the following norms for vectors and matrices: $\vertii{\cdot}$ the Euclidean norm, $\vertif{\cdot}$ the Frobenius norm, $\vertiii{\cdot}$ the spectral norm, and $\vertim{\cdot}$ the entry-wise maximum norm.

Given $M \in \C^{ n \times n}$, we denote by $M^\top$ its transpose, and by $\Tr(M)=\sum_{i=1}^n M_{ii}$ its trace. If $M$ is real and diagonalizable we denote by $\Sp M$ its spectrum and by $\mu_M = \frac 1 n \sum_{\lambda \in \Sp M} \delta_\lambda$ its empirical spectral distribution. If $M$ is symmetric positive semi-definite, then $\Sp M \subset \R^+$, and given a spectral parameter $z \in \C^+ = \{ \omega \in \C$ such that $\Im(\omega) > 0\}$, we define its resolvent $\cal G_M(z) = (M-zI_n)^{-1}$ and its Stieltjes transform $g_M(z) = (1/n) \Tr \cal G_M(z)$.

If $\mu$ is a real probability distribution, its Stieltjes transform is $g_\mu(z) = \int_\R \frac {\mu(dt)} {t-z}  $, well defined for $z \in \C^+$. It is easy to see that the Stieltjes transform of a matrix is the same as the Stieltjes transform of its empirical spectral distribution, that is $g_{\mu_M}(z) = g_M(z) $. We denote by $\cal F_\mu$ be the cumulative distribution function of the measure $\mu$, and by $D(\mu,\nu) = \sup_{t \in \R} \verti{\cal F_\mu(t) - \cal F_\nu(t)}$ the Kolmogorov distance between $\mu$ and $\nu$. 

Let us also introduce two operations on measures: if $\gamma \in \R$, we denote by $\gamma \cdot \mu + (1-\gamma) \cdot \nu$ the new measure such that $\pt{\gamma \cdot \mu + (1-\gamma) \cdot \nu}(B) = \gamma \mu(B) + (1-\gamma) \nu(B)$ for any measurable set $B$ (it may be a signed measure). If $a,b \in \R$, we denote by $a+b\mu$ the distribution of the random variable $a+bX$, where $X$ is a random variable distributed according to $\mu$.

We denote  by  $\cal N$ a standard (centered and reduced) Gaussian random variable, or the law of this random variable depending on the context. We denote by $\delta_a$ the Dirac delta distribution centered at $a\in\R$, and by $\mmp(\gamma)$ the Marčenko-Pastur distribution with shape parameter $\gamma > 0$. The operator $\boxtimes$ denotes the multiplicative free convolution between measures (\cite{bercovici1993free}). The distribution $\mmp(\gamma) \boxtimes \mu$ may be also defined by its Stieltjes transform $g(z)$, which is the unique solution of the self-consistent equation \cite{MP67}:
\begin{align*} g(z) = \int_\R \frac 1 {( 1 - \gamma - \gamma z g(z) ) t - z } \mu(dt).
\end{align*}

For a better readability, we will sometimes omit  to mention indices $n$ and spectral parameters $z$ in our notations, especially in the course of technical proofs, even if we are dealing implicitly with sequences and complex functions.

\newpage

\section{Technical tools} \label{Sec2}
\subsection{Concentration framework}

\begin{defi}\label{DefConc} Let $X_n$ be a sequence of random vectors in finite dimensional normed vector spaces $(E_n,\vertii{\cdot})$, and let $\sigma_n >0$ be a sequence. 
\begin{enumerate}
    \item We say that $X$ is Lipschitz concentrated if there is a constant $C>0$ such that for any sequence of $1$-Lipschitz maps $f_n:E_n\to \C$, for any $n \in \N$ and $t \geq 0$:
\begin{align*}
    \Pp\pt{\verti{f_n(X_n) - \E\br{f_n(X_n)}}\geq t} \leq C e^{- \frac 1 C \pt{\frac {t}{ \sigma_n}}^2}.
\end{align*}
We note this Lipschitz concentration property $X \propto_{\vertii{\cdot}} \cal E(\sigma_n)$, or simply $X \propto \cal E(\sigma_n)$ when there is no ambiguity on the chosen norm.
\item If $Z_n \in E_n$ is a sequence of deterministic vectors, we say that $X$ is linearly concentrated around $Z$ if there is a constant $C>0$ such that for any sequence of $1$-Lipschitz linear maps $f_n:E_n\to \C$, for any $n \in \N$ and $t \geq 0$:
\begin{align*}
    \Pp\pt{\verti{f_n(X_n) - f_n(Z_n)}\geq t} \leq C e^{- \frac 1 C \pt{\frac {t}{ \sigma_n}}^2}.
\end{align*}
We note this linear concentration property $X \in_{\vertii{\cdot}} Z \pm \cal E(\sigma_n)$, or simply $X \in Z \pm \cal E(\sigma_n)$.
\end{enumerate}
\end{defi}

We refer to \cite{LC21} for a comprehensive study of these notions of  concentration. Let us enumerate the main properties that will be used throughout this article:
\begin{prop}\label{PropConc}\begin{enumerate}
\item A random vector (respectively a matrix) with i.i.d. $\cal N$ entries is $\propto \cal E(1)$ concentrated with respect to the Euclidean norm (respectively the Frobenius norm). Other examples are listed in \cite[Theorem 1]{LC21}. \item A Lipschitz transformation of a Lipschitz concentrated vector still verifies a Lipschitz concentration property \cite[Proposition 1]{LC21}.
\item If $X \propto \cal E(\sigma_n)$, $Y \propto \cal E(\sigma_n')$, and $X$ and $Y$ are independent, then $(X,Y) \propto \cal E(\sigma_n+\sigma_n')$ (\cite[Proposition 7] {LC21}).
  \item Lipschitz concentration implies linear concentration: if $X \propto \cal E(\sigma_n)$, then $X \in \E\br{X} \pm \cal E(\sigma_n)$(\cite[Proposition 4] {LC21}). Both definitions are equivalent in one dimension.
  \item If $X \in Z \pm \cal E(\sigma_n)$ and $\tilde Z$ is another deterministic vector, we have the equivalence: $X \in \tilde Z \pm \cal E(\sigma_n)$ $\iff$ $\vertii{Z-\tilde Z} \leq O(\sigma_n)$ (\cite[Lemma 3] {LC21}).
  \item If $X \in Z \pm \cal E(\sigma_n)$ and $X'$ is another random vector such that $\vertii{X'}\leq O(\sigma_n)$ a.s., then $X+X' \in Z \pm \cal E(\sigma_n)$. \item Since $\vertiii{\cdot}\leq\vertif{\cdot}$ on matrices, the concentration with respect to the Frobenius norm is a stronger property than with the spectral norm.
    \item The map $Y \mapsto \pt{Y^\top Y / n - z I_p}^{-1}$ is $\frac{2^{3/2}|z|^{1/2}}{\Im(z)^2 \sqrt n}$-Lipschitz with respect to the Frobenius norm (\cite[Proposition 4.3]{moi1}). In particular, if $Y \propto_{\vertif{\cdot}} \cal E(1)$, then $\cal G_{Y^\top Y / n }(z) \propto_{\vertif{\cdot}} \cal E\pt{\frac{|z|^{1/2}}{\Im(z)^2 \sqrt n}}$. 
       \item The map $M \mapsto \frac 1 n \Tr\pt{M}$ is $1/\sqrt n$-Lipschitz in Frobenius norm. In particular, if $\cal G_K(z) \propto_{\vertif{\cdot}} \cal E(\sigma_n)$ for some random matrix $K \in \R^{n \times n}$, then $g_K(z) \propto \cal E(\sigma_n/\sqrt n)$. 
       \item If a random variable $X$ is $\propto \cal E(\sigma_n) $ concentrated, then $\verti{X-\E\br{X}}\leq O(\sqrt{\log n} \, \sigma_n)$ a.s. (\cite[Proposition 3.3]{moi1}).
\end{enumerate}
\label{prop202}
\end{prop}

We will also need a result for a product of matrices, and introduce for this reason a notion of conditional concentration.   If  $\cal B$ is a measurable subset of the universe $\Omega$ with $\Pp(\cal B)>0$, the random matrix $X$ conditioned with the event $\cal B$, denoted as $\pt{X|\cal B}$, designates the measurable mapping $\pt{\cal B,\cal F_{\cal B},\Pp/\Pp(\cal B)}\mapsto E_n$ satisfying: $\forall \omega \in \cal B$, $\pt{X|\cal B}(\omega) = X(\omega)$.

\begin{prop}[Proposition 9 in \cite{LC20}] \label{ConcProd}
    If $X$ and $Y$ are sequences of random matrices with dimensions bounded by $O(n)$, such that $(X,Y)\propto_{\vertif{\cdot}} \cal E(1) $,  $\E\br{{\vertiii{X}}}\leq O(\sqrt n)$ and $\E\br{{\vertiii{Y}}}\leq O(\sqrt n)$, then there exists a constant $c>0$ such that:
    \begin{enumerate}
        \item $\E\br{\vertiii{X}}\leq c \sqrt n$ and $\E\br{\vertiii{Y}}\leq c \sqrt n$.
        \item With $\cal B = \{\vertiii{X}$ and $\vertiii{Y}\leq 2 c \sqrt n\})$, $\Pp(\cal B^c )\leq ce^{-n/c}$.
        \item $\pt{XY | \cal B} \propto_{\vertif{\cdot}} \cal E(\sqrt n)$.
    \end{enumerate}
\end{prop}

\subsection{Polynomial bounds in $z$ and notation $O_z(\epsilon_n)$}

Throughout this article we will deal with quantitative estimates involving both a dimension parameter $n \in \N$ and a spectral parameter $z \in \C^+$. As it turns out, for many applications it is not useful to track the exact dependence in $z$, which motivates the following definition:

\begin{defi}\label{admseq} Let $\zeta : \N \times \C^+ \to \R^+$ be a function and $\epsilon_n>0$ a sequence. We say that $\zeta$ is bounded by $\epsilon_n$ in $n$ and polynomially in $z$, and we note $\zeta(n,z) \leq O_z\pt{\epsilon_n}$, if there exists $\alpha\geq0$ such that uniformly in $z \in \C^+$ with bounded $\Im(z)$:
\[ \zeta(n,z) \leq O\pt{\epsilon_n \frac{|z|^\alpha}{\Im(z)^{2 \alpha}}}.\]
We say that a family of functions is uniformly $O_z\pt{\epsilon_n}$ bounded if the above bound holds uniformly for any function of the family.\end{defi}

\begin{rem} 

If $\Im(z)$ is bounded, then $\frac{|z|}{\Im(z)}\geq 1$ and $\frac 1 {\Im(z)}$ is bounded away from $0$. As a consequence, if $\alpha \geq \alpha'\geq 0$, then ${\frac{|z|^\alpha}{\Im(z)^{2 \alpha}}}\leq O\pt{\frac{|z|^{\alpha'}}{\Im(z)^{2 \alpha'}}}$. The classical rules of calculus thus apply to our notation: $O_z(\epsilon_n)+O_z(\epsilon_n')=O_z(\epsilon_n+\epsilon_n')$ and $O_z(\epsilon_n)O_z(\epsilon_n')=O_z(\epsilon_n\epsilon_n')$.
\end{rem}

The next technical lemma will be key to translate results like the deterministic equivalent Theorem \ref{DetEqOrig} into simplified versions using our $O_z(\epsilon_n)$ notations.
\begin{lem}\label{technicalLemmaO_z}
    If for some constants $\alpha_0, \alpha_1, \dots, \alpha_6 >0$, the function $\zeta$ satisfies an \emph{a priori} inequality $\verti{\zeta(n,z)}\leq O\pt{\frac{|z|^{\alpha_1}}{\epsilon_n^{\alpha_0}\Im(z)^{{\alpha_1}+{\alpha_2}}}} $, and if the bound: \[ \zeta(n,z) \leq O\pt{\epsilon_n \frac{|z|^{\alpha_3}}{\Im(z)^{{\alpha_3+{\alpha_4}}}}} \] holds uniformly for $z \in \C^+$ with bounded $\Im(z)$ and satisfying $\epsilon_n \frac{|z|^{\alpha_5}}{\Im(z)^{{\alpha_5+{\alpha_6}}}} \leq c$ for some constant $c>0$, then for some  exponent $\alpha>0$, the bound: \[ \zeta(n,z) \leq O\pt{\epsilon_n \frac{|z|^\alpha}{\Im(z)^{2 \alpha}}} \] holds uniformly in $z \in\C^+$ with bounded $\Im(z)$, that is $\zeta \leq O_z(\epsilon_n)$.
\end{lem}

\begin{proof} Let us choose $\alpha=\max\pt{\alpha_5(\alpha_0+1) + \alpha_1, \alpha_6(\alpha_0+1)+\alpha_2,\alpha_3 ,\alpha_4}$.     If $\epsilon_n \frac{|z|^{\alpha_5}}{\Im(z)^{{\alpha_5+{\alpha_6}}}} \leq c$, we use the main inequality on $\zeta$:\begin{align*}
        \zeta(n,z) \leq O\pt{\epsilon_n \frac{|z|^{\alpha_3}}{\Im(z)^{{\alpha_3+{\alpha_4}}}}} \leq O\pt{\epsilon_n \frac{|z|^\alpha}{\Im(z)^{2 \alpha}}}.
    \end{align*}
 In the case where $\epsilon_n \frac{|z|^{\alpha_5}}{\Im(z)^{{\alpha_5+{\alpha_6}}}} \geq c$, we use the \emph{a priori} inequality on $\zeta$:
    \begin{align*}
      \verti{\zeta(n,z)}&\leq O\pt{\frac{|z|^{\alpha_1}}{\epsilon_n^{\alpha_0}\Im(z)^{{\alpha_1}+{\alpha_2}}}} \\
      &\leq O\pt{\epsilon_n \pt{\frac{|z|^{\alpha_5}}{c \,\Im(z)^{{\alpha_5+{\alpha_6}}}}}^{\alpha_0+1} \frac{|z|^{\alpha_1}}{\Im(z)^{{\alpha_1}+{\alpha_2}} }}\\
      &\leq O\pt{ \epsilon_n \frac{|z|^\alpha}{\Im(z)^{2\alpha}}}.
    \end{align*}
\end{proof}

\begin{rem}
    As we will see later in this article, the lemma is particularly suited to simplify some quantitative statements  that require the spectral parameter $z$ to be away from the real axis. Indeed the Stieltjes tranforms and resolvent matrices satisfy the classical bounds $\verti{g(z)}\leq 1 / \Im(z)$ and $\vertif{\cal G(z)}\leq \sqrt n \vertiii{\cal G(z)}\leq \sqrt n / \Im(z)$, which translate into \emph{a priori} inequalities for any difference of such objects.
\end{rem}Let us wrap this subsection by giving a general setting on which $O_z(\epsilon_n) $ bounds between Stieltjes transforms imply polynomial bounds in Kolmogorov distance. We remind our reader that for measures $\nu$ and $\mu$ with cumulative distribution functions  $\cal F_\nu$ and $\cal F_\mu$, the Kolmogorov distance is defined as $D(\mu,\nu) = \sup_{t \in \R}\verti{F_{\nu}(t) - \cal F_{\mu}(t)} $.
The following result is an immediate consequence of \cite[Theorem 3.6]{moi1}.
\begin{prop}\label{OzandKol}
Let $\mu_n$ and $\nu_n$ be sequences of probability measures on $\R^+$ such that:
\begin{enumerate}
    \item $\int_{\R} \verti{  \cal F_{\mu_n}(t) - \cal F_{\nu_n}(t)}dt < \infty $.
    \item $\cal F_{\nu_n}$ are uniformly Hölder continuous with exponent $\beta \in (0,1]$.
    \item $\mu_n$ and $\nu_n$ have uniformly bounded second moments.
    \item $\verti{g_{\mu_n}(z)-g_{\nu_n}(z)}\leq O_z(\epsilon_n)$ for some sequence $\epsilon_n \to 0$. \end{enumerate}
    Then there exists $\theta >0$ such that $D(\mu_n,\nu_n) \leq O(\epsilon_n^\theta)$.\end{prop}
\begin{rem}\label{remHoldtrick}
In some cases, the uniform Hölder continuity hypothesis may be adapted using a simple change of measures. For instance, given shape parameters $\gamma_n > 1$ and  measures $\tau_n$ supported on the same compact of $(0,\infty)$, the cumulative distribution function of $\nu_n = \mmp(\gamma_n) \boxtimes \tau_n$ are not even continuous at $0$. However we can consider the new measures $\check \nu_n = (1-\gamma_n) \cdot \delta_0 + \gamma_n \cdot \nu_n$, and use the fact that $\cal F_{\check \nu_n}$ are uniformly $1/2$-Hölder continuous to deduce the same result (see \cite[Section 8.3]{moi1}).

A variant of this result for empirical spectral distributions  may also be adapted, see Proposition \ref{thmSigmaLKol} in the Appendix.
\end{rem}

\newpage

\section{Covariance matrices of functions of Gaussian vectors}\label{Sec3}

 In this section we consider the random vector $y=f(u) \in \R^n$, obtained by applying a real function on each of the coordinates of a centered Gaussian vector $u \in \R^n$. We are particularly interested in the case where the entries of $u$ are weakly correlated. In this instance, we will give
 approximations of the  covariance matrix $\Sigma=\E\br{y y^\top}$, with a control of the error either entry-wise, or in spectral norm, or for their respective Stieltjes transforms. We first present the tools making these approximations possible.

\subsection{Hermite polynomials}  Let us remind the notation $\cal N$ for a standard Gaussian random variable. We denote by $\cal H$ the Hilbert space of real Borel functions such that $\E\br{f(\cal N)^2}<\infty$, endowed with the Gaussian inner product: 
\[<f,g> = \E\br{f(\cal N) g(\cal N)}= \frac 1 {\sqrt{2 \pi}}\int_\R f(t)g(t)e^{-t^2/2 } dt. \] Remark that all Lipschitz continuous functions automatically belong to $\cal H$. The $r$-th non-normalized Hermite polynomial is by definition:
\[h_r(t)=(-1)^r e^{t^2/2 }\frac{\rm{d}^r}{\rm{d}t^r}\pt{e^{-t^2/2 }}.\] The first Hermite polynomials are: $h_0 = 1$, $h_0 = \bf X$, $h_2 = \bf X^2 - 1$, and $h_3 = \bf X^3 - 3\bf X$. 
The $r$-th normalized Hermite polynomial is $\hat h_r = \frac {h_r}{\sqrt{r!}}$, and the $r$-th Hermite coefficient of a function $f \in \cal H$ is $\zeta_r(f)=<f,\hat h_r>$. Hereafter we remind some classical properties of Hermite polynomials without proofs. For  more details, we invite our reader to consult \cite[Chapter 4]{sansone1959orthogonal}.

\begin{prop}\label{hermite_nol_classic_nrop}
\begin{enumerate}[(i)]
    \item  $h_r$ are monic polynomials of degree $r$. $\hat h_r$ form a complete orthonormal basis of $\cal H$.
   \item Every function of $\cal H$ can be expanded as the converging sum in $\cal H$: $f = \sum_{r \geq 0} {\zeta_r(f)}  \hat h_r$. In particular $\vertih{f} ^2 = \E\br{f(\cal N)^2}=\sum_{r\geq 0} {\zeta_r(f)^2}$. \item The Hermite polynomials satisfy the relations:
    \begin{align*}
        h_k'(t)&=kh_{k-1}(t),\\
        h_{r+1}(t)&=t h_r(t)-h_r'(t).
    \end{align*}    
\end{enumerate}
\end{prop}

We move on to more specific properties that shall be used in the course of this section. 

\begin{lem}[Lemma D.2 in \cite{nguyen2020global}]\label{hermite_nol_lem1} If $z_1$ and $z_2$ are standard Gaussian random variables, such that $(z_1,z_2)$ forms a Gaussian vector, then:
\[\E\br{\hat h_r(z_1) \hat h_s(z_2)}=\bf 1_{r=s} \Cov[z_1,z_2]^r.\]
\end{lem}

\begin{prop}\label{hermite_nol_lem2} 
    Let us fix $f \in \cal H$, and for $r \in \N$ let $\Psi_r (\sigma)= \sigma^{-r}\E\br{f\pt{\sigma \cal N }h_r(\cal N)}$.  Then $\Psi_r$ is $\cal C^\infty$ on $(0,\infty)$ and $\Psi_r'(\sigma)=\sigma\Psi_{r+2}(\sigma)$.
\end{prop}

\begin{proof} From the identity
    $\E\br{f(\sigma \cal N )h_r(\cal N)} = \frac 1 {\sigma \sqrt{2 \pi}}\int_\R f(t) h_r(t/\sigma) e^{-\frac{t^2}{2 \sigma^2}} dt$  and the Leibniz integral rule, we deduce that $\Psi_r$ is $\cal C^\infty$ and that:
    \begin{align*}
         \Psi_r'(\sigma) &= \frac{d}{d\sigma} \pt{\frac 1 {\sigma^{r+1} \sqrt{2 \pi}}\int_\R f(t) h_r(t/\sigma) e^{-\frac{t^2}{2 \sigma^2}} dt } \\
         &= \frac 1 {\sqrt{2 \pi}} \int_\R f(t)  \frac{d}{d\sigma} \pt{\frac{1}{\sigma^{r+1}} h_r(t/\sigma) e^{-\frac{t^2}{2 \sigma^2}}} dt.
    \end{align*}
The derivative inside the integral can be computed in a few steps. Firstly from the definition $\frac d {dt} \pt{h_r(t) e^{-t^2/2 }} = (-1)^r \frac { d^{r+1}}{dt^{r+1}}\pt{e^{-t^2/2}}= - h_{r+1}(t)  e^{-t^2/2 }$. Then applying the identity $t h_{r+1}(t)=h_{r+2}(t)+(r+1) h_{r}(t)$, which can be deduced from Proposition \ref{hermite_nol_classic_nrop}, leads to:
\begin{align*}
   \frac{d}{d\sigma} \pt{ h_r(t/\sigma) e^{-\frac{t^2}{2 \sigma^2}}} &= \frac t {\sigma^2} h_{r+1}(t/\sigma) e^{-t^2/2\sigma^2 } \\
    &= \frac 1 {\sigma} \pt{h_{r+2}(t/\sigma)+(r+1)h_r(t/\sigma)}e^{-\frac{t^2}{2 \sigma^2}},\\
     \frac{d}{d\sigma} \pt{\frac{1}{\sigma^{r+1}} h_r(t/\sigma) e^{-\frac{t^2}{2 \sigma^2}}}  &=  {\frac{1}{\sigma^{r+2}} \pt{h_{r+2}(t/\sigma)+(r+1)h_r(t/\sigma)} } e^{-\frac{t^2}{2 \sigma^2}} - \frac{r+1}{\sigma^{r+2}} h_r(t/\sigma)  e^{-\frac{t^2}{2 \sigma^2} }\\
        &= \frac{1}{\sigma^{r+2}} h_{r+2}(t/\sigma) e^{-\frac{t^2}{2 \sigma^2}}.
\end{align*} 
We obtain finally:
\begin{align*}
\Psi_r'(\sigma) &=\frac 1 {\sqrt{2 \pi}} \int_\R f(t) \frac{1}{\sigma^{r+2}} h_{r+2}(t/\sigma) e^{-\frac{t^2}{2 \sigma^2}}  dt \\
         &= \sigma \cdot \pt{\frac 1 {\sigma^{r+3} \sqrt{2 \pi}}\int_\R f(t) h_{r+2}(t/\sigma) e^{-\frac{t^2}{2 \sigma^2}} dt } \\
         &= \sigma \cdot \Psi_{r+2}(\sigma).
    \end{align*}
\end{proof}

\begin{cor}\label{Expansionzeta}
    Let $f \in \cal H$, and for $\sigma >0$ let $f_\sigma : t \mapsto f(\sigma t)$. Then for any integer $r \geq 0$:
    \begin{align*}
        \zeta_r(f_\sigma) 
        &= \sigma^r \pt{\zeta_r(f) + O(\sigma-1)}\\&= \sigma^r \pt{ \zeta_r(f) + \sqrt{(r+1)(r+2)} (\sigma - 1)\zeta_{r+2}(f) + O((\sigma - 1)^2)}.
    \end{align*}   
\end{cor}

\begin{proof} These expressions are straightforward consequences of first and second order Taylor expansions of $\Psi_r$. For instance for the second formula:
    \begin{align*}
        \zeta_r(f_\sigma)
        &= \frac{\sigma^r}{\sqrt{r!}} \Psi_r(\sigma) \\
        &= \frac{\sigma^r}{\sqrt{r!}} \pt{\Psi_r(1) + (\sigma - 1) \Psi_{r+2}(1) + O((\sigma - 1)^2) } \\
        &= \sigma^r \pt{ \zeta_r(f) + \sqrt{(r+1)(r+2)} (\sigma - 1)  \zeta_{r+2}(f) + O((\sigma - 1) ^2)}.
    \end{align*}
\end{proof}

\subsection{Iterated Hadamard products}

In what follows, $M^{\circ r}$ denotes the Hadamard product of $r$ copies of a matrix $M \in \R^{n \times n}$. We denote by $\rm{diag}(M)$ and $\rm{off}(M)$ respectively the diagonal and off-diagonal sub-matrices of $M$. $\vec{\rm{diag}}(M) \in \R^n$ denotes the diagonal elements of $M$ reshaped as a vector in $\R^n$.

\begin{lem} \label{CorrOffDelta} If $\Delta \in \R^{n \times n}$ is a symmetric matrix such that $I_n + \Delta$ is positive semi-definite, then for any integer $r \geq 1$: 
\begin{align*}     \vertiii{\Delta^{\circ r+1}} &\leq \pt{1+\vertii{\Delta}_{\max}} \vertiii{\Delta^{\circ r}}  + \vertii{\Delta}_{\max}^r  , \\
     \vertiii{\Delta^{\circ 2r}} &\leq \vertii{\Delta}_{\max} ^{2r-2}   \vertiii{\Delta \circ \Delta }  .   
\end{align*}
    
\end{lem}
\begin{proof} For any matrices $A$ and $B \in \R^{n\times n}$ with $A$ positive semi-definite, the inequality $\vertiii{A \circ B} \leq \vertii{A}_{\max} \vertiii{B}$ holds true (\cite[Proposition 3.7.9.]{johnson1990matrix}). As a consequence:
\begin{align*}
    \vertiii{\Delta^{\circ r+1}} &\leq  \vertiii{(I_n+\Delta) \circ \Delta^{\circ r}} + \vertiii{I_n \circ \Delta^{\circ r}} \\
    &\leq \vertii{I_n+\Delta }_{\max} \vertiii{\Delta^{\circ r}}  + \vertiii{\rm{diag}\pt{\Delta}^r} \\
    &\leq \pt{1+ \vertii{\Delta }_{\max} }\vertiii{\Delta^{\circ r}}  + \vertii{\Delta}_{\max}^r  .\end{align*}
    
 For the second inequality, since $\Delta^{\circ 2r}$ has non negative entries, there is $\alpha \in \pt{\R^+}^n$ such that $\vertii{\alpha}=1$ and $\vertiii{\Delta^{\circ 2r}} = \vertii{\Delta^{\circ 2r} \alpha}$. We deduce that:
\begin{align*}
\vertiii{\Delta^{\circ 2r}}  &= \pt{ \sum_{i=1}^n \pt{\sum_{j=1}^n \Delta_{ij}^{2r} {\alpha_j}}^2}^{1/2} \\
&\leq  \vertii{\Delta}_{\max}^{2r-2} \pt{ \sum_{i=1}^n \pt{\sum_{j=1}^n\Delta_{ij}^{2} {\alpha_j}}^2\,}^{1/2} \\
&\leq  \vertii{\Delta}_{\max} ^{2r-2} \vertiii{\Delta \circ \Delta} .
\end{align*}
\end{proof}

\begin{prop}  \label{SpecHadam}If $\Delta \in \R^{n \times n}$ is a sequence of symmetric matrices such that $I_n + \Delta$ is positive semi-definite, $\vertii{\Delta}_{\max}$ converges to $0$ and  $\vertiii{\Delta}$ is bounded, then for any $r \geq 1$ the quantities  $\vertiii{\Delta^{\circ 2r}}$ , $\vertiii{\rm{off}(\Delta^{\circ 2r})}$, $\vertiii{\Delta^{\circ 2r+1}}$ and $\vertiii{\rm{off}(\Delta^{\circ 2r+1})}$ are all bounded by $ O\pt{\vertii{\Delta}_{\max}^{2r-2}}$ when $n$ goes to $\infty$.
\end{prop}

\begin{proof}
    According to the preceding lemma, $
        \vertiii{\Delta \circ \Delta} \leq \pt{1+\vertii{\Delta}_{\max}} \vertiii{\Delta}  + \vertii{\Delta}_{\max}      \leq O(1)$. This implies that: $\vertiii{\Delta^{\circ 2r}}\leq \vertii{\Delta}_{\max} ^{2r-2}   \vertiii{\Delta \circ \Delta } \leq O\pt{\vertii{\Delta}_{\max}^{2r-2}}$, and $\vertiii{\Delta^{\circ 2r+1}}\leq \pt{1+\vertii{\Delta}_{\max}} \vertiii{\Delta^{\circ 2r}}  + \vertii{\Delta}_{\max}^{2r} \leq  O\pt{\vertii{\Delta}_{\max}^{2r-2}}  $. The remaining results follow after remarking that $\vertiii{\rm{diag}(\Delta^{\circ r}) }\leq O\pt{\vertii{\Delta}_{\max}^r}$.\end{proof}

\subsection{General expansion of the covariance matrix $\Sigma$}

In the following paragraphs we are considering a function $f \in \bf H$, and $u \in \R^n$ a centered Gaussian vector with covariance matrix $S \in \R^{n \times n}$. The random vector $y=f(u)$ is obtained by applying the function $f$ entry-wise on $u$, and we let $\Sigma = \E\br{y y^\top}$. 
For $1 \leq i \leq n$, we denote by $f_i$ the function $f_i : t \mapsto f\pt{ S_{ii}^{1/2} \, t } $, and by $D_r \in \R^{n \times n}$ the diagonal matrix with entries ${{S_{ii}}^{-r/2}} {\zeta_r(f_i)}$. 

\begin{prop}\label{GeneralExpansionSigma} The convergence of the following sum holds entry-wise: \[ \Sigma = \sum_{r \geq 0}  D_r\,{ S^{\circ r} }D_r .\]
\end{prop}

\begin{proof} The random variables $\tilde u_i = S_{ii}^{-1/2} u_i$ are standard Gaussian random variables, and they form a Gaussian vector. Using Lemma \ref{hermite_nol_lem1}, we have:\begin{align*}
    \Sigma_{ij} &= \E\br{f_i(\tilde u_i)f_j(\tilde u_j)} \\ &= \sum_{r,s\geq0} {\zeta_r(f_i)}{\zeta_s(f_j)}\E\br{\hat h_r(\tilde u_i )\hat h_s(\tilde u_j)} \\
    &= \sum_{r\geq 0} {\zeta_r(f_i)\zeta_r(f_j) }\Cov[\tilde u_i,\tilde u_j]^r \\
    &= \sum_{r\geq 0} {\zeta_r(f_i)\zeta_r(f_j) } {S_{ii}}^{-r/2} {S_{jj}}^{-r/2} {S_{ij}}^r \\
    &= \sum_{r\geq 0}  \pt{D_r}_i\pt{D_r}_j (S^{\circ r})_{ij},
\end{align*} 
which proves the Proposition.
\end{proof}

\subsection{Approximation of $\Sigma$ for weakly correlated Gaussian vectors}
We will now explain how the expansion of $\Sigma$ given by Proposition \ref{GeneralExpansionSigma} can be simplified when $u$ has weakly correlated entries, in the sense that its covariance matrix $S$ is close to $I_n$. To this end we let $\Delta = S - I_n$, and:
\begin{align*}
  \Sigma_{\rm{approx}}= \vertih{f}^2 I_n + \frac {\zeta_2(f)^2}{2} \vec{\rm{diag}}(\Delta)\vec{\rm{diag}}(\Delta)^\top + \sum_{r=1,2,3}\zeta_r(f)^2 \Delta^{\circ r}
\end{align*}

\begin{ass} \label{AssFuncGaussVector2} 
\begin{enumerate}
    \item $f$ is Lipschitz and Gaussian centered, i.e. $\zeta_0(f)=\E\br{f(\cal N)}=0$.
     \item $\vertiii{\Delta}$ and $\vertii{\vec{\rm{diag}}(\Delta)}$ are bounded.
      \item $\vertii{\Delta}_{\max}$ converges to $0$.
\end{enumerate}\end{ass}
\begin{thm} \label{SigmaApprox}Under Assumptions    \ref{AssFuncGaussVector2}, $\vertiii{\Sigma - \Sigma_{\rm{approx}}}\leq O\pt{\vertii{\Delta}_{\max} }$. \end{thm}

\begin{proof}
Let $\epsilon_n = \vertii{\Delta}_{\max}$. Since $f$ is Lipschitz, $\vertii{f - f_i}_{L^2} = O\pt{\verti{1-S_{ii}^{1/2 }}}=O(\epsilon_n)$. We deduce that $\vertiii{\rm{diag}(\Sigma) - \vertii{f}I_n} = \max_{1 \leq i \leq n} \verti{\vertii{f_i} - \vertih{f}} \leq O(\epsilon_n)$. For the off-diagonal terms, let $\nu$ be the vector in $\R^n$ defined by its coordinates  $\nu_i=\zeta_0(f_i)$. Let also $\mu = \frac {\zeta_2(f)^2}{2} \vec{\rm{diag}}(\Delta)$. The decomposition:
\[\rm{off}(\Sigma) = \rm{off}\pt{\nu \nu^\top } + \sum_{r\geq 1} D_r \,    \rm{off}(\Delta^{\circ r})\, D_r\] follows easily from the expansion of $\Sigma$ given by Proposition \ref{GeneralExpansionSigma}. Since $\zeta_0(f)=0$ and $S_{ii}^{1/2} = 1 + \frac {\Delta_{ii}} 2 + O(\Delta_{ii}^2)$, using the second order expansion given by Corollary \ref{Expansionzeta} applied to the function $f$, uniformly in $i \in \lint 1, n\rint$ we have:
\begin{align*}
    \nu_i = \zeta_0(f_i)&=\sqrt 2 \pt{S_{ii}^{1/2} - 1}  \zeta_2(f) + O\pt{\pt{S_{ii}^{1/2} - 1}^2} \\
    &= \frac{\zeta_2(f)}{\sqrt 2}   \Delta_{ii} + O(\Delta_{ii}^2) \\
    &= \mu_i + O(\Delta_{ii}^2).
\end{align*} Therefore $\vertii{\nu - \mu}^2 =O\pt{ \sum_{i=1}^n  \Delta_{ii}^4}\leq O \pt{\max_{1 \leq i \leq n} \Delta_{ii}^2 } \vertii{\vec{\rm{diag}}(\Delta)}^2 \leq O(\epsilon_n^2)$, and:
\begin{align*}
    \vertiii{\rm{off}\pt{\nu \nu^\top} - \mu \mu^\top } &\leq \vertii{\nu-\mu}\pt{\vertii{\nu}+\vertii{\mu}}+\vertiii{ \rm{diag}\pt{\nu \nu^\top} }\leq O(\epsilon_n)
\end{align*}

For the remaining sum $\sum_{r\geq 1} D_r \,    \rm{off}(\Delta^{\circ r})\, D_r$, we use again that $\verti{\zeta_r(f_i)-\zeta_r(f) } \leq \vertih{f_i-f}\leq O(\epsilon_n)$, hence $\vertiii{D_r - \zeta_r(f) I_n} \leq O(\epsilon_n)$, and in particular $\vertiii{D_r}\leq O(1)$. Using Proposition \ref{SpecHadam}, for $r = 1,2$ or $3$:
\begin{align*}
    \vertiii{D_r \,    \rm{off}(\Delta^{\circ r})\, D_r - \zeta_r(f)^2 \Delta^{\circ r} }& \leq \vertiii{D_r}^2 \vertiii{\rm{diag}(\Delta^{\circ r})} \\
    &\quad + \vertiii{{D_r - \zeta_r(f) I_n}} \vertiii{\Delta^{\circ r}} \pt{ \vertiii{D_r} + \verti{\zeta_r(f)}  }  \\
    &\leq O(\epsilon_n).
\end{align*} Finally for the remaining terms,  using Proposition \ref{SpecHadam} for $r \geq 4$:
\begin{align*}
  \vertiii{\sum_{r\geq 4} D_r \,    \rm{off}(\Delta^{\circ r})\, D_r} &\leq \sum_{r \geq 2} O\pt{\epsilon_n^{2r-2}} \leq O\pt{\epsilon_n^2}.
\end{align*}
\end{proof}

\subsection{Linearization of $\Sigma$ in specific settings} Let:
\[ \Sigma_{\rm{lin}} = \vertih{f}^2 \, I_n + \zeta_1(f)^2 \Delta \]
The matrix $\Sigma_{\rm{lin}}$ is obtained by means of usual matrix operations, easy to handle with classical random matrix theory tools, in contrast to $\Sigma_{\rm{approx}}$ which involves Hadamard products. $\vertiii{ \Sigmaa - \Sigmal}$ may not converge to $0$ under the mere Assumptions \ref{AssFuncGaussVector2}. We can nonetheless identify conditions involving $\Delta$ and $f$, under which $\Sigmal$ is a good approximation of $\Sigma$ either entry-wise, or in spectral norm, or for their respective Stieltjes transforms.

\begin{prop} \label{SigmaLinNorm}Under Assumptions  \ref{AssFuncGaussVector2}, $\vertiii{\Sigma - \Sigmal}$ is bounded, and with $\epsilon_n = \zeta_2(f)^2 \vertim{\Delta}^2 + \zeta_3(f)^2 \vertim{\Delta}^3$:\begin{align*}
    \vertim{ \Sigma - \Sigmal}&\leq O\pt{\vertii{\Delta}_{\max}},\\
    \vertiii{\Sigma - \Sigmal} &\leq O\pt{\vertii{\Delta}_{\max}  + n \epsilon_n} ,\\
    \verti{g_\Sigma(z) - g_{\Sigmal}(z) } &\leq   \frac 1 {\Im(z)^2}{O\pt{\vertii{\Delta}_{\max}  + \sqrt n \epsilon_n}} .
\end{align*} \end{prop}

\begin{proof} We start from the expression:
\[ \Sigmaa - \Sigmal =\frac {\zeta_2(f)^2}{2} \vec{\rm{diag}}(\Delta)\vec{\rm{diag}}(\Delta)^\top + \zeta_2(f)^2 \Delta^{\circ 2}+\zeta_3(\tilde f)^2 \Delta^{\circ 3}\] The matrices $\vec{\rm{diag}}(\Delta)\vec{\rm{diag}}(\Delta)^\top$, $\Delta^{\circ 2}$ and $\Delta^{\circ 3}$ are all bounded in spectral norm, thus $\vertiii{\Sigmaa - \Sigmal}$ is bounded and from Theorem \ref{SigmaApprox} $\vertiii{\Sigma - \Sigmal}$ is always bounded. The bound $ \vertim{ \Sigmaa - \Sigmal}\leq O\pt{\epsilon_n} $ is immediate given the above expression for the difference of these matrices, and from Theorem \ref{SigmaApprox}:
\begin{align*}
    \vertim{ \Sigma - \Sigmal}&\leq \vertiii{\Sigma-\Sigmaa}+\vertim{ \Sigmaa - \Sigmal} \\
    &\leq O\pt{\vertii{\Delta}_{\max}  + \epsilon_n}\leq O\pt{\vertii{\Delta}_{\max}}.
\end{align*}The other bounds follow from classical matrix inequalities. Indeed: 
\begin{align*}
  \vertiii{\Sigma - \Sigmal} &\leq   \vertiii{\Sigma - \Sigmaa} +   \vertiii{\Sigmaa - \Sigmal} \\
  &\leq O\pt{\vertii{\Delta}_{\max}}   + n \vertim{\Sigmaa - \Sigmal}\\
    &\leq O\pt{\vertii{\Delta}_{\max}  + n \epsilon_n} .
\end{align*}

Finally for the inequality on Stieltjes transforms, for any $A,B \in \R^{n \times n}$ and $z \in \C^+$:
\begin{align*}
    \verti{g_A(z) - g_B(z) } &= \verti{ \frac 1 n \Tr\pt{\cal G_A(z) (B-A) \cal G_B(z)}} \\
    &\leq \frac 1 n \vertiii{\cal G_A(z) } \vertif{I_n} \vertif{B-A}\vertiii{\cal G_B(z) } \\
    &\leq \frac {\vertif{B-A}} {\sqrt n \Im(z)^2}  .
\end{align*}
We deduce that:
\begin{align*}
    \verti{g_\Sigma(z) - g_{\Sigmal}(z) } &\leq  \verti{g_\Sigma(z) - g_{\Sigmaa}(z) } +  \verti{g_\Sigmaa(z) - g_{\Sigmal}(z) } \\
    &\leq \frac {\vertif{\Sigma-\Sigmal}  } {\sqrt n \Im(z)^2}   +  \frac {\vertif{\Sigmaa-\Sigmal}  } {\sqrt n \Im(z)^2}   \\
    &\leq \frac {\vertiii{\Sigma-\Sigmal} } {\Im(z)^2} + \sqrt n  \frac {\vertim {\Sigmaa-\Sigmal} } { \Im(z)^2}  \\
     &\leq \frac 1 {\Im(z)^2} {O\pt{\vertii{\Delta}_{\max}  + \sqrt n \epsilon_n}} .
\end{align*}
   \end{proof}

\begin{rem} \label{SigmaLinNormCor} In practice, $\vertiii{\Sigma - \Sigmal}$ converges to $0$ provided one of the following additional hypothesis holds true:
\begin{itemize}
\item $\vertii{\Delta}_{\max}=o\pt{n^{-1/2}}$,
\item $\zeta_2(f)=0$ and $\vertii{\Delta}_{\max}=o\pt{n^{-1/3}}$,
\item $\zeta_2(f)=\zeta_3(f)=0$.
\end{itemize} 

For the Stieltjes transforms, $g_\Sigma(z) - g_{\Sigmal}(z)$ converges to $0$ pointwise, and thus $\Sigma$ and  $\Sigmal$ have the same limiting spectral distribution if it exists, provided one of the following additional hypothesis holds true:
\begin{itemize}
\item $\vertii{\Delta}_{\max}=o\pt{n^{-1/4}}$,
\item $\zeta_2(f)=0$ and $\vertii{\Delta}_{\max}=o\pt{n^{-1/6}}$,
\item $\zeta_2(f)=\zeta_3(f)=0$.
\end{itemize}
\end{rem}

\newpage

\section{Deterministic equivalent of sample covariance matrices} \label{Sec4}

In this section we recall the latest results about the deterministic equivalent of the resolvents of sample covariance matrices on which this article is based. These estimates were first established in \cite{LC21} with a convergence speed in $n$, and complemented with quantitative estimates in the spectral parameter $z$ in \cite{moi1}. In a second step, we will thoroughly study the properties of the deterministic equivalent matrices appearing in these results.

Given $\Sigma \in \R^{n \times n}$ a positive semi-definite matrix and a sequence of shape parameters $\gamma_n> 0$, we build the matrix function $\bf G_\boxtimes ^{ \Sigma} : \C^+ \to \C^{n \times n}$ from the following objects: \begin{align*}
    \nu^\Sigma &= \mmp\pt{\gamma_n} \boxtimes \mu_\Sigma, \\
    \check \nu^\Sigma &= (1-\gamma_n ) \cdot\delta_0 + \gamma_n\cdot  \nu^\Sigma ,\\
    l_{\check \nu^\Sigma}(z) &= -1/g_{\check \nu^\Sigma}(z) ,\\
    \bf G_\boxtimes ^{ \Sigma}(z) &= \pt{- z g_{\check \nu^\Sigma}(z) \Sigma - z I_n}^{-1}, \\
    &= z^{-1} l_{\check \nu^\Sigma}(z)  \cal G_L\pt{  l_{\check \nu^\Sigma}(z)  }.
\end{align*}

Let us state without proofs some useful properties of these objects. $\check \nu^\Sigma$ is always a true  probability measure (\cite[Lemma 6.1]{moi1}). The usual resolvent inequality $\vertiii{\bf G_\boxtimes ^\Sigma (z)} \leq 1 /\Im(z)$ holds (\cite[Lemma 5.1]{moi1}), and $\Tr\bf G_\boxtimes ^\Sigma(z) =n g_{\nu^\Sigma}(z)$ (\cite[Proposition 6.2]{moi1}). 

\begin{rem} The notation $\bf G_\boxtimes^\Sigma(z)$ is inspired from the free probability theory, which appears as the profound canvas hidden within the above definitions. Indeed it is not difficult to see that the pair $\pt{g_{\check \nu^\Sigma}(z), \bf G_\boxtimes^\Sigma(z)}$ satisfies the system of self-consistent equations:
    \begin{align*}
   g_{\check \nu^\Sigma}(z) &= \pt{-z -  z \frac {\gamma_n} n \Tr\pt{\bf G_\boxtimes^\Sigma(z) \Sigma }}^{-1} , \\
        \bf G_\boxtimes^\Sigma(z) &= \pt{ - z I_n - z g_{\check \nu^\Sigma}(z) \Sigma}^{-1}.
    \end{align*}   
    This may be rewritten as the operator-valued self-consistent equation: \[ z \cal H(z) = I_{n+1} + z \eta\pt{\cal H(z)} \cal H(z),\] where $\cal H(z) = \begin{pmatrix} g_{\check \nu^\Sigma}(z) & 0 \\ 0 & \bf G_\boxtimes^\Sigma(z) \end{pmatrix}$, and: \begin{align*}
     \eta: \C \oplus \C^{n \times n} &\to \C \oplus \C^{n \times n}, \\\begin{pmatrix} g & 0 \\
0 & G \end{pmatrix} &\mapsto \begin{pmatrix} \frac{ \gamma_n} n \Tr\pt{G \Sigma} & 0 \\
0 & g \Sigma \end{pmatrix} .
\end{align*} $z \cal H(z^2)$ thus corresponds to the resolvent of an operator-valued free semi-circular variable with covariance $\eta$, see \cite[Section 3.3]{far2006spectra}. In this sense the definition of $\bf G_\boxtimes^\Sigma(z)$ extends the notion of free convolution, and it should come as no surprise that such objects appear as deterministic equivalent of sample covariance matrices.\end{rem}

\subsection{General results} 
Let $Y \in \R^{d \times n}$ be a sequence of random matrices. The associated sample covariance matrix is $K = Y^\top Y/d$, and for $z \in \C^+$ we define its resolvent $\cal G_K(z) = \pt{K-zI_n}^{-1}$ and Stieltjes transform $g_K(z) =(1/n)  \Tr \cal G_{K}(z) $. From the expected covariance matrix $\Sigma = \E\br{K}$, we follow the above procedure to build the matrix function $\bf G_\boxtimes ^{ \Sigma}(z)$.

\begin{ass} \label{DetEqAss}
\begin{enumerate}
    \item  $Y\propto_{\vertif{\cdot}} \cal E(1)$, the rows of $Y$ are i.i.d. sampled from the distribution of a random vector $y$, with $\vertii{\E[y]}$ and $\vertiii{\E[y y^\top]}=\vertiii{\Sigma}$ bounded.
      \item The ratio $\gamma_n=  \frac  n d $ is bounded from above and away from $0$.
\end{enumerate}
\end{ass}

\begin{thm}{\cite[Theorem 2.3]{moi1}}\label{DetEqOrig} Uniformly under Assumptions \ref{DetEqAss}, there exists $C>0$ such that uniformly in $z \in \C^+$ with $\Im(z)$ bounded and $\frac{ |z|^{ 7 }}{\Im(z)^{16}}{} \leq n/C$, the following inequality holds:
\[ \vertif{\E[\cal G_K(z)] - \bf G_\boxtimes ^{\Sigma}(z) } \leq O\pt{ \frac{ |z|^{5/2}}{\Im(z)^9 n^{1/2}} }.\]
\end{thm}

\begin{rem} By "{uniformly under Assumptions \ref{DetEqAss}}", we mean that the implicit constants in the result only depend on the constants chosen in the assumptions. Said otherwise, if a family of random matrices satisfies Assumptions \ref{DetEqAss} uniformly, then the results holds uniformly for any matrix of this family. 

Keeping track of the exponents appearing in this theorem will become more and more difficult as we go further into this article. We choose for this reason to work with a weaker notion of approximation using our concept of $O_z(\epsilon_n)$ polynomial bounds. As a side effect, we will see that we do not need anymore to specify conditions on $z$ on which the approximation holds true.\end{rem} 

\begin{thm}[Simplified version of the Deterministic Equivalent]\label{DetEqSimpl} Uniformly under Assumptions \ref{DetEqAss}, the following concentration properties hold true:
\begin{enumerate}
    \item $g_K(z) \propto  \cal E\pt{O_z(1/ n)}$ and $\cal G_K(z) \propto_{\vertif{\cdot}}   \cal E\pt{O_z(1/\sqrt n)}$.
    \item $g_K(z) \in g_{\nu}(z) \pm \cal E\pt{O_z(1/n)}$ and $\cal G_K(z) \in_{\vertif{\cdot}}  \bf G_\boxtimes ^{\Sigma}(z)  \pm \cal E\pt{O_z(1/\sqrt n)}$.
\end{enumerate}
\end{thm}
\begin{proof} From Theorem \ref{DetEqOrig}, Lemma \ref{technicalLemmaO_z} and the \emph{a priori} bound $\vertif{\E[\cal G_K(z)] - \bf G_\boxtimes ^{\Sigma} (z) } \leq 2 \sqrt n / \Im(z) $, we have $\vertif{\E[\cal G_K(z)] - \bf G_\boxtimes ^{\Sigma}(z) } \leq O_z(1/\sqrt n)$. The general properties of concentration recalled in Proposition \ref{PropConc} prove the Theorem.
\end{proof}

Similarly to \cite[Corolaries 2.5 and 2.8]{moi1}, we may deduce the following spectral properties of $K$. We will not prove this result immediately, but rather prompt our reader to consult the proof of Corollary \ref{CorDetAnn} which is extremely similar.

\begin{cor} \label{DetEqCor}Uniformly under Assumptions \ref{DetEqAss} :
\begin{enumerate}
\item    $\verti{g_K(z) - g_{\nu^\Sigma}(z)} \leq O_z\pt{ \sqrt{\log n}/n}$ and  $\vertim{{\cal G_K (z)- \bf G_\boxtimes^\Sigma(z) }}  \leq O_z\pt{ \sqrt{\log n/n}}$. 
   \item     
If the eigenvalues of $\Sigma$ are bounded from below, there exists $\theta >0$ such that $D(\mu_K, \nu^\Sigma) \leq O(n^{-\theta}) $  a.s. 
\item  If additionally $\mu_\Sigma \to \mu_\infty$ weakly and $\gamma_n \to \gamma_\infty$, then $\mu_K \to \nu_\infty = \mmp(\gamma_\infty) \boxtimes \mu_\infty$  weakly a.s., and more precisely: 
      \[ D(\mu_K, \nu_\infty) \leq D(\mu_\Sigma,\mu_\infty) + O(|\gamma_n-\gamma_\infty|)+O(n^{-\theta}) \quad \rm{a.s.}\]   
\end{enumerate}

\end{cor}

\subsection{Regularity of the Stieltjes transform with respect to the free convolution}

Let ${\Sigma}\in \R^{n \times n}$ be a sequence of positive semi-definite matrices and $\tau$ probability measures such that $g_{\Sigma}(z)-g_{\tau}(z)$ converges to $0$ pointwise in $\C^+$. Then $\mu_{\Sigma}$ and $\tau$ converge weakly to the same limit if it exists. The free multiplicative convolutions $\nu= \mmp({\gamma_n})\boxtimes \mu_{\Sigma}$ and $\chi = \mmp({\gamma_n}) \boxtimes \tau$ also converge weakly to the same limit, which is equivalent to  their Stieltjes transforms $g_{\nu}(z)$ and $g_{\chi}(z)$ having the same limit pointwise. In this section we would like to refine this result by quantifying the convergence of $g_{\nu}(z) - g_{\chi}(z)$ to $0$. 

In the upcoming paragraphs, we always consider shape parameters $\gamma_n$ that are bounded from above and away from $0$, like in Assumptions \ref{DetEqAss}. We remind our reader that we are dealing with sequences of matrices, measures, and complex functions, even if we sometimes omit the indices $n$ and $z$ for a better readability.

\begin{thm} \label{72} Let ${\Sigma} \in \R^{n \times n}$ be deterministic positive semi-definite matrices, and $\tau$ deterministic probability measures supported on $\R^+$.

If $\verti{g_{\Sigma}(z)-g_\tau(z)}\leq O_z(\epsilon_n)$, then $\verti{g_{\nu}(z)-g_{\chi}(z)} \leq O_z(\epsilon_n)$.
\end{thm}

The proof of this result may be decomposed in several steps. First we translate the definition of the measures $\nu$ and $\chi$ into appropriate self-consistent equations on their reciprocal Cauchy transforms (Proposition \ref{prop63}). Then we see that $\chi$ is an approximate fixed point of the equation corresponding to $\nu$ (Proposition \ref{prop65}). Finally we use the stability of these self-consistent equations and the tools developed in \cite{LC21} and \cite{moi1} to conclude.

The key function in the upcoming paragraphs  
$ l_{\check \nu^{\Sigma}}(z)=-1/g_{\check \nu}(z)$ is known as the reciprocal Cauchy transform of the measure $\tilde \nu^{\Sigma}$. As such some classical properties of this function may be found in the seminal book \cite[Section 3.4]{mingo2017free}. We will nonetheless provide short proofs for all the properties we need in this article.

    \begin{prop}[Self-consistent equation for reciprocal Cauchy transforms]\label{prop63} If  $\mu$ is a measure supported on $\R^+$, we let $\nu=\mmp({\gamma_n}) \boxtimes \mu$, $\check \nu = (1-{\gamma_n}) \cdot \delta_0 + {\gamma_n} \cdot \nu$, and $l_{\check \nu}(z) = -1/g_{\check \nu}(z) $.

    Then
        $l_{\check \nu}(z)$  is the only solution on $\C^+$ of the self consistent equation in $l$:
        \[ l = z + {\gamma_n} l + {\gamma_n} l^2 g_\mu(l) = z + {\gamma_n} z \int_{\R} \frac{t}{\frac {z t} l - z} \mu(dt) .
        \]
    \end{prop}

    \begin{proof} The two right hand side terms  are always equal since:
    \begin{align*}
        \int_{\R} \frac{zt}{\frac {z t} l - z} \mu(dt) &=  l \int_{\R} \frac{t}{t-l} \mu(dt) = l \int_{\R} \pt{1 +  \frac{l}{t-l} }\mu(dt) = l + l^2 g_\mu(l).
    \end{align*}Let us work with the first formulation. It is a classical property of Stieltjes transforms that $g_{\check \nu}(z) \in \C^+$ when $z \in \C^+$, hence $l_{\check \nu}(z) \in \C^+$. Since $g_{\check \nu} = \frac{{\gamma_n} - 1} z + {\gamma_n} g_{\nu}$, we have the identity ${1-{\gamma_n}-{\gamma_n} z  g_{\nu}} = \frac z {l_{\check \nu}}$. By definition of the free multiplicative convolution with a Marčenko-Pastur distribution, $g_{\nu} $ is the only solution on $\C^+$ of:
    \begin{align*}
        g_{\nu} &= \int_\R \frac 1 {\pt{1-{\gamma_n}-{\gamma_n} z  g_{\nu}} t - z} \mu(dt) \\
        &= \int_\R \frac 1 {z t/ l_{\check \nu} - z} \mu(dt) \\
        &= \frac {l_{\check \nu}} z \int_\R \frac 1 {t -  l_{\check \nu} } \mu(dt)
        \\  &= \frac {l_{\check \nu}} z  g_\mu\pt{l_{\check \nu}}. 
    \end{align*}
  Using again the identity ${1-{\gamma_n}-{\gamma_n} z  g_{\nu}} = \frac z {l_{\check \nu}}$, the self-consistent equation characterizing $  g_{\nu} $ is equivalent for $ l_{\check \nu}$ to satisfy:
    \begin{align*}
        {\gamma_n} l_{\check \nu} +   {\gamma_n} l_{\check \nu}^2g_\mu\pt{l_{\check \nu}}&=   {\gamma_n}  l_{\check \nu}\pt{ 1 +    z g_{\nu}}=  l_{\check \nu} - z.
    \end{align*}
    \end{proof}

\begin{lem}\label{lem66}With the same notations and hypothesis as in Proposition \ref{prop63}:
\begin{align*}
  0 &\leq   \Im \pt{z^{-1}l_{\check \nu}(z)}, \\\Im(z) &\leq \Im \pt{l_{\check \nu}(z)}\leq O(1+\Im(z)) ,\\ \verti{l_{\check \nu}(z)}&\leq O\pt{\frac{|z|}{\Im(z)}}. 
\end{align*}

In particular if a function $\zeta : \N \times \C^+ \to \R$ satisfies $\zeta(n,z) \leq O_z(\epsilon_n)$, then $\zeta(n,l_{\check \nu}(z)) \leq O_z(\epsilon_n) $. \end{lem}

\begin{proof} $\nu=\mmp({\gamma_n}) \boxtimes \mu$ is supported on $\R^+$, thus: 
  \begin{align*}
    -   \Im\pt{\frac z {l_{\check \nu}}} &= - \Im\pt{1-{\gamma_n}-{\gamma_n} z g_\nu} \\
      &=  {\gamma_n}  \int_{\R^+} \Im\pt{\frac{z}{t-z}} \nu(dt) \\
      &= {\gamma_n} \Im(z) \int_{\R^+} \frac t {|t-z|^2} \nu(dt) \geq 0,
  \end{align*}
  which proves that $\Im\pt{z^{-1} l_{\check \nu}} \geq 0$. Secondly:
  \begin{align*}
      \Im\pt{l_{\check \nu} - z } &=  \Im\pt{{\gamma_n} z \int_{\R} \frac{t}{\frac {z t} {l_{\check \nu}} - z} \mu(dt)  } \\
      &= {\gamma_n}  \int_{\R^+} \Im\pt{ \frac{1}{1/l_{\check \nu}-1/t} } \mu(dt)  \\
      &= {\gamma_n}   \int_{\R^+} \frac{ \Im\pt{{l_{\check \nu}}} t^2}{  \Re\pt{{l_{\check \nu}}}^2 + \Im\pt{{l_{\check \nu}}}^2+ t^2 - 2  \Re\pt{{l_{\check \nu}}} t} \mu(dt) \geq 0 .
  \end{align*} The right hand side integral is bounded from above by $\frac 1 {\Im\pt{l_{\check \nu}}}\int_{\R^+}  t^2 \mu(dt)$, hence $\Im\pt{l_{\check \nu}} ^2\leq \Im(z)\Im\pt{l_{\check \nu}} + {\gamma_n} \int_{\R^+}  t^2 \mu(dt)$. Solving this second order polynomial inequality gives $\Im\pt{l_{\check \nu}} \in \br{ \Im(z)/2 \pm \pt{{\gamma_n} \int_{\R^+}  t^2 \mu(dt)+\Im(z)^2/4}^{1/2}}$ and $\Im\pt{l_{\check \nu}}  \leq O\pt{1+\Im(z)}$. We also have $\Im\pt{\frac {zt} {l_{\check \nu}} } \leq 0$ for any $t \geq 0$, hence:
\begin{align*}
 \verti{l_{\check \nu}}  &\leq |z| + {\gamma_n} |z| \int_{\R^+} \frac{t}{\verti{\Im\pt{\frac{zt}{l_{\check \nu}} - z}}} \mu(dt) \\
 &\leq |z| + {\gamma_n} \frac{|z|}{\Im(z)} \int_{\R^+} t\, \mu(dt) \leq O\pt{\frac{|z|}{\Im(z)}}.\end{align*}

For the last statement, let $\zeta $ be a function such that $\zeta(n,z) \leq O_{z}(\epsilon_n)  $. If $\Im(z)$ is bounded, so is $\Im\pt{l_{\check \nu}(z)}$, hence there exists $\alpha >0$ such that:
\[\zeta(n,l_{\check \nu}(z))\leq O\pt{\epsilon_n \frac{\verti{l_{\check \nu}(z)}^\alpha}{\Im\pt{l_{\check \nu}(z)}^{2\alpha}}}\leq O\pt{\epsilon_n \frac{|z|^\alpha}{\Im(z)^{3\alpha}}}\leq O\pt{\epsilon_n \frac{|z|^{2\alpha}}{\Im(z)^{4\alpha}}},\] from which we deduce that $\zeta(n,l_{\check \nu}(z))\leq O_z(\epsilon_n)$.
\end{proof}
 We may now move on to the second part of the proof of Theorem \ref{72}. Let us define the mapping: \[    \cal F : l \in \C^+ \mapsto z +  {\gamma_n} \frac z n \Tr{\pt{\pt{\frac z l {\Sigma} - z I_n}^{-1} {\Sigma} }}. \]  As shown in Proposition \ref{prop63}, the definition of $\cal F$ is equivalent to:
\[\cal F(l)  = z + {\gamma_n} z \int_{\R} \frac{t}{\frac {z t} l - z} \mu_{\Sigma}(dt)  = z + {\gamma_n} l + {\gamma_n} l^2 g_{\Sigma}(l) .\]
We set $l_{\check \nu}(z)=-1/g_{\check \nu}(z)$ and $l_{\check \chi}(z)=-1/g_{\check \chi}(z)$. In Proposition \ref{prop63} we have proved that $l_{\check \nu}(z)$ is a fixed point of $\cal F$. We will see in the next Proposition that $l_{\check \chi}(z)$ is almost a fixed point of $\cal F$.

    \begin{prop} \label{prop65} $
    \verti{\cal F(l_{\check \chi}(z)) - l_{\check \chi} (z)} \leq  O_z\pt{\epsilon_n}$.
    \end{prop}
    \begin{proof}  Using the first formulation of the self-consistent equation of Proposition \ref{prop63},  we have $l_{\check \chi} = z +  {\gamma_n} _n l_{\check \chi} + {\gamma_n}_n l_{\check \chi}^2 g_{\nu} (l_{\check \chi})$, thus:
    \[ \cal F(l_{\check \chi}) - l_{\check \chi}  =  {\gamma_n}_n l_{\check \chi}^2 \pt{ g_{\mu_{\Sigma}}(l_{\check \chi})   - g_{\tau} (l_{\check \chi}) } .\]
As seen in Lemma \ref{lem66}, since $\verti{g_{\mu_{\Sigma}}(z)-g_\tau(z)}\leq O_z(\epsilon_n)$, we also have $ \verti{ g_{\mu_{\Sigma}}(l_{\check \chi} (z))   - g_{\tau} (l_{\check \chi}(z) ) }\leq O_z(\epsilon_n)$, and we obtain:
\begin{align*}    \verti{\cal F(l_{\check \chi}) - l_{\check \chi} } &\leq \verti{{\gamma_n}_n} \verti{l_{\check \chi}}^2  \verti{ g_{\mu_{\Sigma}}(l_{\check \chi} )   - g_{\tau} (l_{\check \chi} ) } \\
    &\leq O(1) O\pt{\frac{|z|}{\Im(z)}}^2 O_z(\epsilon_n)\leq O_z(\epsilon_n).
\end{align*}
    \end{proof}

The last step to prove Theorem \ref{72} is to use the stability of the self-consistent equation $l=\cal F(l)$.  Let us recall the tools and results established established in \cite[ Section 6]{moi1}.  For a fixed $z\in \C^+$, we introduce the domain $\bf D = \{ \omega \in \C$ such that $\Im(\omega) \geq \Im(z)$ and $\Im(z^{-1}\omega) \geq 0 \}$, and the semi-metric on $\C^+$:
    \[ d(\omega_1,\omega_2) = \frac{|\omega_1-\omega_2|}{\Im(\omega_1)^{1/2}\Im(\omega_2)^{1/2}}.\]

$\cal F $ is a contraction mapping on $\bf D$ with respect to $d$. More precisely, $\cal F$ is $k_{\cal F}$-Lipschitz with $k_{\cal F}= \frac {\frac{|z|}{\Im(z)^2}}{1+\frac{|z|}{\Im(z)^2}}$ (\cite[Proposition 6.11]{moi1}). Moreover, if $c \in \bf D$ is a fixed point of $\cal F$ and $b \in \bf D$ any other point, provided $k_{\cal F}\pt{1+d\pt{b,\cal F(b)}}<1$, the following inequality holds true (\cite[Lemma 6.14]{moi1}):
        \[ \verti{c-b} \leq \frac{\verti{\cal F(b)-b}}{1-k_{\cal F}\pt{1+d\pt{b,\cal F(b)}}}.\]

We have now collected all the arguments required to compare the Stieltjes transforms of $\nu$ and $\chi$.

\begin{proof}[Proof of Theorem \ref{72}] If $\Im(z)$ is bounded, using Proposition \ref{prop65} there exists $\alpha $ and $C>0$ such that $\verti{\cal F(l_{\check \chi}) - l_{\check \chi} } \leq C {\epsilon_n \frac{|z|^\alpha}{\Im(z)^{2\alpha}}}$. For values of $z$ such that $\epsilon_n \frac{|z|^{\alpha+1}}{\Im(z)^{2\alpha+3}}  \leq 1/2C$, $\frac{|z|}{\Im(z)^2}d\pt{\cal F( l_{\check \chi}) ,l_{\check \chi} } \leq 1/2$, thus:\begin{align*}
    k_{\cal F}\pt{1+d\pt{\cal F( l_{\check \chi}) , l_{\check \chi}}} &=\frac {\frac{|z|}{\Im(z)^2} + {\frac{|z|}{\Im(z)^2} d\pt{\cal F( l_{\check \chi}) , l_{\check \chi}}}}{1+\frac{|z|}{\Im(z)^2}}  \\
    &\leq 1 - \frac 1 {2\pt{{1+\frac{|z|}{\Im(z)^2}}}}.
\end{align*}
In particular  $k_{\cal F}\pt{1+d\pt{\cal F( l_{\check \chi}) , l_{\check \chi}}}<1$, and from \cite[Lemma 6.14]{moi1}:
\begin{align*}
    \verti{l_{\check \chi} - l_{\check \nu} } &\leq \frac{{ \verti{\cal F( l_{\check \chi}) - l_{\check \chi}}} }{1-k_{\cal F}\pt{1+d\pt{\cal F( l_{\check \chi}) , l_{\check \chi}}}} \\
    &\leq  2\pt{1+\frac{|z|}{\Im(z)^2}} C {\epsilon_n \frac{|z|^\alpha}{\Im(z)^{2\alpha}}} \\
    &\leq O\pt{\epsilon_n \frac{|z|^{ \alpha+1}}{\Im(z)^{2 \alpha + 2}}} .\end{align*}
To conclude:
\begin{align*}
        \verti{g_{\nu}- g_{\chi}} = \frac {\verti{g_{\check \nu} - g_{\check \chi}}} {\gamma}  = \frac { |g_{\check \nu}| |g_{\check \chi}| |l_{\check \nu}-l_{\check \chi}|} {{\gamma_n}} \leq O\pt{\epsilon_n \frac{|z|^{ \alpha+1}}{\Im(z)^{2 \alpha + 2}}},
\end{align*}
for values of $z$ such that  $\epsilon_n \frac{|z|^{\alpha+1}}{\Im(z)^{2\alpha+3}}  \leq \frac 1{2C }$. Given the \emph{a priori} bound $ \verti{g_{\nu}(z)- g_{\chi}(z)} \leq \frac 2 {\Im(z)}$ and Lemma \ref{technicalLemmaO_z}, the bound $\verti{g_{\nu}(z)- g_{\chi}(z)} \leq O_z(\epsilon_n)$ holds true.\end{proof}

\subsection{Approximation of deterministic equivalent built from deterministic matrices}

Given an approximation for the Stieltjes transform $g_{\Sigma}(z) \approx g_\tau(z)$, and another approximation for the resolvent $\cal G_{\Sigma}(z) \approx \bf H(z)$, can we find an approximation for the matrix $\bf G_\boxtimes ^{ {\Sigma}}(z)$ ? We build a matrix function $\bf K$ using the same procedure we used earlier to build  $\bf G_\boxtimes ^{ {\Sigma}}$  from $\mu_{\Sigma}$ and $\cal G_{\Sigma}$:
\begin{align*}
    \chi &= \mmp\pt{{\gamma_n} } \boxtimes \tau ,\\
    \check \chi  &= (1-{\gamma_n} ) \cdot\delta_0 + {\gamma_n} \cdot \chi ,\\
    l_{\check \chi}(z) &= -1/g_{\check \chi}(z), \\
     \bf K(z) &= z^{-1} l_{\check \chi}(z)  \bf H\pt{l_{\check \chi}(z)  }.
\end{align*}

\begin{prop} \label{ApproxbgGDet} Let ${\Sigma} \in \R^{n \times n}$ be deterministic positive semi-definite matrices, $\tau$ deterministic probability measures supported on $\R^+$, and $\bf H:\C^+ \to \C^{n \times n}$ deterministic complex functions.

If $\verti{g_{\Sigma}(z)-g_\tau(z)}\leq O_z(\epsilon_n)$ and $\vertiii{\cal G_{\Sigma}(z)-\bf H(z)}\leq O_z(\epsilon_n')$, then $\vertiii{ \bf G_\boxtimes^  {\Sigma}(z) - \bf K(z)} \leq O_z(\epsilon_n+\epsilon_n')$.
\end{prop}

\begin{proof}
We use a triangular inequality in the following decomposition:
\begin{align*}
     \bf G_\boxtimes^  {\Sigma}(z) - \bf K(z) &= z^{-1} {  \pt{ l_{\check \nu}  -   l_{\check \chi}   } \cal G_{\Sigma}\pt{l_{\check \nu}}}  \\ 
     &\quad + z^{-1} l_{\check \chi}  \pt{ \cal G_{\Sigma}\pt{l_{\check \nu}} - \cal G_{\Sigma}\pt{l_{\check \chi}} }  \\
     &\qquad  + z^{-1} l_{\check \chi}   \pt{ \cal G_{\Sigma}\pt{l_{\check \chi} } -  \bf H\pt{l_{\check \chi} } }.
\end{align*}
For the first term, $\verti {  { l_{\check \nu}  -   l_{\check \chi}   }} \leq  O_z\pt{\epsilon_n }$ and $\vertiii{\cal G_{\Sigma}\pt{l_{\check \nu}}}  \leq \frac 1 {\Im\pt{ l_{\check \nu}}}\leq \frac 1 {\Im(z)}$, thus $\vertiii{z^{-1} {  \pt{ l_{\check \nu}  -   l_{\check \chi}   } \cal G_{\Sigma}\pt{l_{\check \nu}}}} \leq O_z\pt{\epsilon_n }$. For the second term, using a resolvent identity:
\begin{align*}
    \vertiii{z^{-1} l_{\check \chi}  \pt{ \cal G_{\Sigma}\pt{l_{\check \nu}} - \cal G_{\Sigma}\pt{l_{\check \chi}} } } &\leq \verti{z^{-1} l_{\check \chi}(z)} \vertiii{\cal G_{\Sigma}\pt{l_{\check \nu}}}  \verti {  { l_{\check \nu}  -   l_{\check \chi}   }} \vertiii{\cal G_{\Sigma}\pt{l_{\check \chi}}} \\
    &\leq\frac1{|z|} O\pt{\frac {|z|}{\Im(z)}}  {\frac {O_z\pt{\epsilon_n }} {\Im(z)^2}}  \leq O_z\pt{\epsilon_n}.
\end{align*}
Finally for the third term, $\vertiii{ \cal G_{\Sigma}\pt{l_{\check \chi} } -  \bf H\pt{l_{\check \chi} } } \leq O_z(\epsilon_n')$ using Lemma \ref{lem66}, thus:
\begin{align*}
 \vertiii{z^{-1} l_{\check \chi}   \pt{ \cal G_{\Sigma}\pt{l_{\check \chi} } -  \bf H\pt{l_{\check \chi} } } } \leq \frac1{|z|} O\pt{\frac {|z|}{\Im(z)}}    O_z(\epsilon_n') \leq O_z(\epsilon_n').
\end{align*}\end{proof}

\begin{cor}\label{GboxLip} The map ${\Sigma} \mapsto \bf G_\boxtimes ^{\Sigma}(z)$ is $O_z(1)$ Lipschitz with respect to the spectral norm.
\end{cor}

\begin{proof} Using a resolvent identity we have:
\begin{align*}
   \vertiii{ \cal G_{\Sigma}(z) - \cal G_{{\Sigma}'}(z)} \leq  \vertiii{ \cal G_{\Sigma}(z)} \vertiii{{\Sigma}' - {\Sigma}}\vertiii{ \cal G_{\Sigma}(z)} \leq \frac{\vertiii{{\Sigma} - {\Sigma}'}}{\Im(z)^2}. \end{align*} In particular $\verti{g_{\Sigma}(z)-g_{{\Sigma}'(z)}}\leq \vertiii{ \cal G_{\Sigma}(z) - \cal G_{{\Sigma}'}(z)}\leq O_z\pt{\vertiii{{\Sigma} - {\Sigma}'}}$. The result follows from Theorem \ref{72} and Proposition \ref{ApproxbgGDet}. 
\end{proof}

\subsection{Concentration of deterministic equivalents built from random matrices}
If ${\Sigma}$ is random and satisfies a typical $\cal E(1/\sqrt n)$ Lipschitz concentration property, from the approximations $\E \br{g_{\Sigma}(z)}\approx g_\tau(z) $ and $\E\br{\cal G_{\Sigma}(z)} \approx \bf H(z)$, we may deduce that $g_{\Sigma}(z) \approx g_\tau(z)$ a.s., but we cannot expect that $\cal G_{\Sigma}(z) \approx \bf H(z)$ since $\cal G_{\Sigma}(z)$ and $\E\br{\cal G_{\Sigma}(z)}$ are not necessarily close in spectral norm. We can however prove that $\bf G_\boxtimes^{\Sigma}(z)$ is linearly concentrated around $\bf K(z)$.

\begin{prop} \label{ApproxbgGRnd} Let ${\Sigma} \in \R^{n \times n}$ be random positive semi-definite matrices such that ${\Sigma} \propto_{\vertif{\cdot}} \cal E(1/\sqrt n)$. Then: \begin{enumerate}
    \item $g_{\nu}(z) \propto \cal E(O_z(1/ n))$ and $\bf G_\boxtimes^{\Sigma}(z)\propto_{\vertiii{\cdot}} \cal E(O_z(1/\sqrt n))$.
    \item If $\tau$  are deterministic probability measures supported on $\R^+$ such that $\verti{\E\br{g_{\Sigma}(z)}-g_\tau(z)}\leq O_z(\epsilon_n)$, then $\verti{g_\nu(z)-g_\chi(z) }\leq O_z(\epsilon_n+\sqrt{\log n} / n)$ a.s. 
    \item If in addition $\bf H:\C^+ \to \C^{n \times n}$ are deterministic complex functions such that $\vertiii{\E\br{\cal G_{\Sigma}(z)}-\bf H(z)}\leq O_z(\epsilon_n')$, then $\bf G_\boxtimes^  {\Sigma}(z) \in_{\vertiii{\cdot}} \bf K(z) \pm \cal E(O_z(\epsilon_n+\epsilon_n'+1/\sqrt n))$.
\end{enumerate} 
\end{prop}

\begin{proof} We refer to Proposition \ref{PropConc} for the properties of Lipschitz and linear concentration used in this proof. The map ${\Sigma} \mapsto  \bf G_\boxtimes^  {\Sigma} $ is $O_z(1)$ Lipschitz with respect to the spectral norm, thus $\bf G_\boxtimes^  {\Sigma} \propto_{\vertiii{\cdot}} \cal E(O_z(1/\sqrt n))$. Remembering the identity $\Tr \bf G_\boxtimes ^{\Sigma} = n g_{\nu}$, we also deduce that  $g_{\nu} \propto \cal E(O_z(1/\sqrt n))$. 

The map ${\Sigma} \mapsto \cal G_{\Sigma} $ is $1/\Im(z)^2$ Lipschitz, thus $\cal G_{\Sigma} \propto_ {\vertif{\cdot}} \cal E(O_z(1/\sqrt n))$ and $g_{\Sigma} \propto \cal E(O_z(1/ n))$. We deduce that $ \verti{\E\br{g_{\Sigma}}-g_{\Sigma} }\leq O_z(\sqrt{\log n} / n)$ a.s., hence $\verti{g_{\Sigma}-g_\tau }\leq O_z(\epsilon_n+\sqrt{\log n} / n)$ a.s. We can apply Theorem \ref{72} uniformly in this set of full measure, and obtain that $\verti{g_\nu-g_\chi }\leq O_z(\epsilon_n+\sqrt{\log n} / n)$ a.s.

In the decomposition:
\begin{align*}
     \bf G_\boxtimes^  {\Sigma}(z) - \bf K(z) &= z^{-1} {  \pt{ l_{\check \nu}  -   l_{\check \chi}   } \cal G_{\Sigma}\pt{l_{\check \nu}}}  \\ 
     &\quad + z^{-1} l_{\check \chi}  \pt{ \cal G_{\Sigma}\pt{l_{\check \nu}} - \cal G_{\Sigma}\pt{l_{\check \chi}} }  \\
     &\qquad  + z^{-1} l_{\check \chi}   \pt{ \cal G_{\Sigma}\pt{l_{\check \chi} } -  \bf H\pt{l_{\check \chi}  }},
\end{align*}the first two terms are bounded by $O_z(\epsilon_n+\sqrt{\log n} / n)\leq O_z(\epsilon_n+1 / \sqrt n )$ a.s. in spectral norm (see the proof of Proposition \ref{ApproxbgGDet}). For the third term, $\cal G_{\Sigma} (z) \propto_ {\vertif{\cdot}} \cal E(O_z(1/\sqrt n))$ and $\vertiii{\E\br{\cal G_{\Sigma}(z)}-\bf H(z)}\leq O_z(\epsilon_n')$, thus $\cal G_{\Sigma} (z)\in_ {\vertiii{\cdot}} \bf H (z) \pm \cal E(O_z(\epsilon_n' + 1/\sqrt n))$. From Lemma \ref{lem66}, we also have $ \cal G_{\Sigma}\pt{l_{\check \chi} } \in_{\vertiii{\cdot}}   \bf H\pt{l_{\check \chi}  } \pm \cal E(O_z(\epsilon_n' + 1/\sqrt n))$. Finally $\verti{z^{-1} l_{\check \chi}}\leq 1/\Im(z)$, and combining the above estimates and concentration properties leads to $\bf G_\boxtimes^  {\Sigma} (z)\in_{\vertiii{\cdot}}   \bf K(z) \pm \cal E(O_z(\epsilon_n+\epsilon_n'+1/\sqrt n))$. \end{proof}

\newpage

\section{Single-layer neural network with deterministic data} \label{Sec5}

\subsection{Setting} 
In this section we consider the Conjugate Kernel matrix associated to a single-layer artificial neural network with deterministic input. The model is made of:
\begin{itemize}
    \item a random weight matrix $W \in \R^{d \times d_0}$, with variance parameter $\sigma_W^2>0$,
    \item a deterministic data matrix $X \in \R^{ d_0 \times n}$, $\sigma_X^2 > 0$,
    \item two random biases matrices $B$ and $D \in \R^{d \times n }$, $\sigma_B^2,\sigma_D^2 \geq 0$,
    \item and an activation function $f : \R \to \R$.
\end{itemize} As output of the neuron, we set $Y =f(WX/\sqrt{d_0}+B)+D \in \R^{d \times n}$, where the function $f$ is applied entry-wise. Our goal is to investigate the spectral properties of the Conjugate Kernel matrix $K= Y^\top Y / d$. For $z \in \C^+$ we define its resolvent $\cal G_{K}(z)=\pt{K - z I_n}^{-1}$ and Stieltjes transform $g_K(z) =(1/n)  \Tr \cal G_{K}(z) $. We also define the following objects:
\begin{align*}      \tilde \sigma^2 &= \sigma_W^2 \sigma_X^2 + \sigma_B^2 , \\
    \tilde f (t)&= f(\tilde \sigma t), \\ 
    \frak a &=  {\vertih{\tilde f}^2  -  \frac{\sigma_W^2\sigma_X^2}{ \tilde \sigma^2} \zeta_1(\tilde f)^2 + \sigma_D^2 }, \\
    \frak b &= \zeta_1(\tilde f)^2 \frac{\sigma_W^2}{\tilde \sigma^2},\\K_X &= X^\top X /{ d_0}, \\ \Delta_X &= K_X- \sigma_X^2 I_n, \\ \Sigma &= \E\br{K}, \\
 \Sigmal &= \frak a I_n +  \frak b K_X.
\end{align*}
\begin{rem} We always have $\frak a \geq  0$ since $\zeta_1(f)^2 \leq \vertih{\tilde f}^2 $ and $\sigma_W^2\sigma_X^2 \leq \tilde \sigma^2$. Moreover $\frak a=0$ if and only if $f$ is a linear function and there is no bias in the model (that is $\sigma_B^2=\sigma_D^2=0$).\end{rem} Like in the rest of this article, we sometimes omit the indices $n$ and $z$ for a better readability, even if we are implicitly dealing with sequences of matrices, measures, and complex functions.

\begin{ass} \label{ANNDetAss}
\begin{enumerate} \item $W$, $B$ and $D$ are random, independent, with i.i.d. $\cal N(\sigma_W^2)$, $\cal N(\sigma_B^2)$ and $\cal N(\sigma_D^2)$ entries respectively.     \item $\tilde f$ is Lipschitz continuous and Gaussian centered, that is $\E\br{\tilde f(\cal N)}=\E\br{f(\tilde \sigma \cal N)}=0$.
    \item $X$ is deterministic and $\vertiii{K_X}$ is bounded.
     \item $\vertii{\vec{\rm{diag}}(\Delta_X)}$ is bounded and $\vertim{\Delta_X} $ converges to $0$.
     \item The ratio ${\gamma_n}= \frac n d$ is bounded from above and away from $0$.
\end{enumerate}
\end{ass} 

We refer to the Remark \ref{ANNDetAssrem1}  for a detailed discussion about the assumption (4).
The main result 
of this section is Theorem \ref{ANNDetDetEq} that gives a deterministic equivalent for $\cal G_K(z)$ and $g_K(z)$. To prove this result, we will combine the general results on resolvent matrices recalled in Section \ref{Sec4}, with the linearization techniques of Section \ref{Sec3}. We will wrap up this section by applying our framework to a simple yet original model having weakly correlated entries.

\subsection{Technicalities and linearization of $\Sigma$}   

\begin{prop} \label{propa} Under Assumptions \ref{ANNDetAss}, $Y \propto_{\vertif{\cdot}} \cal E(1)$. The rows of $Y$ are i.i.d. sampled from the distribution of a random vector $y=f(X^\top w/\sqrt {d_0} +b)+\tilde b$, where $w \in \R^{d_0}$ and $b,\tilde b \in \R^n$ are independent Gaussian vectors with i.i.d. $\cal N(\sigma_W^2)$, $\cal N(\sigma_B^2)$ and $\cal N(\sigma_D^2)$ coordinates respectively. $\vertii{\E\br{y}}$ is moreover bounded.
\end{prop}

\begin{proof}    
The map $(W,B,D) \mapsto f(WX/\sqrt {d_0}+B)+D$ is Lipschitz with respect to the product Frobenius norm since $f$ is Lipschitz and $\vertiii{X/\sqrt {d_0}}$ is bounded. Given the Gaussian concentration $(W,B,D) \propto_{\vertif{\cdot}} \cal E(1)$, we immediately obtain that $Y \propto_{\vertif{\cdot}} \cal E(1)$. From the expression:
\[Y_{ij} = f\pt{\sum_{k=1}^{d_0} W_{ik}X_{kj}/\sqrt {d_0}+B_{ij}}+D_{ij},\] we see that the rows of $Y$ are independent and have the same distribution as $y = f\pt{X^\top w /\sqrt {d_0}+ b}+\tilde b$.  For the last statement, we may write $y = \tilde f(u)+\tilde b$, where $u$ is a centered Gaussian vector with covariance matrix $S = \pt{\sigma_W^2 K_X + \sigma_B^2 I_n}/{\tilde \sigma^2}$. We have $S - I_n = (\sigma_W^2 / \tilde \sigma^2) \Delta_X$, hence the  random variables $u_i$ are centered Gaussian with covariance $1 + O\pt{(\Delta_X)_{ii}}$. Since $\zeta_0(\tilde f)=\E\br{f(\tilde \sigma \cal N)}=0$, using the first order expansion given by Corollary \ref{Expansionzeta} applied to the function $\tilde f$, we have $\E\br{\tilde f(u_i)}=O\pt{(\Delta_X)_{ii}}$ uniformly on $i \in \lint 1 , n \rint$. We deduce that $\vertii{\E\br{y}} =   \vertii{\E\br{\tilde f(u)}} =  O\pt{\vertii{\vec{\rm{diag}}(\Delta_X)}}=O(1)$.
\end{proof}

\begin{cor} \label{propb}Under Assumptions \ref{ANNDetAss}, $  \vertiii{\Sigma }$ is bounded, and moreover:
\begin{align*}
    \vertiii{\Sigma - \Sigmal } &\leq O(\vertim{\Delta_X}+n  \zeta_2(\tilde f)^2 \vertim{\Delta_X}^2 +n \zeta_3(\tilde f)^2 \vertim{\Delta_X}^3 ), \\
    \verti{g_\Sigma(z)-g_\Sigmal(z)}&\leq  O(\vertim{\Delta_X}+\sqrt n  \zeta_2(\tilde f)^2 \vertim{\Delta_X}^2 +\sqrt n \zeta_3(\tilde f)^2 \vertim{\Delta_X}^3 ).
\end{align*}
\end{cor}
\begin{proof} As seen in the last Proposition, the rows of $Y$ are i.i.d. sampled from the distribution of a vector $y = \tilde f(u)+\tilde b$ where $u$ is a centered Gaussian vector with covariance matrix $S = I_n + \Delta$, and $\Delta = (\sigma_W^2 / \tilde \sigma^2) \Delta_X$. The Assumptions \ref{AssFuncGaussVector2} are satisfied, and we deduce from Proposition \ref{SigmaLinNorm} that $  \vertiii{\Sigma - \Sigmal }$ is bounded. Since $\vertiii{K_X}$ is bounded and $\Sigmal = \frak a I_n +  \frak b K_X$, $  \vertiii{\Sigma }$ is also bounded. From the same proposition, we have the following estimates in spectral norm, with $\epsilon_n  =\vertim{\Delta_X}+n  \zeta_2(\tilde f)^2 \vertim{\Delta_X}^2 +n \zeta_3(\tilde f)^2 \vertim{\Delta_X}^3 $:
\begin{align*}
    \E\br{\tilde f(u) \tilde f(u)^\top } &= \vertii{\tilde f}^2 I_n + \zeta_1(\tilde f)^2 \Delta+ O(  \epsilon_n) \\
    &= \vertii{\tilde f}^2 I_n + \zeta_1(\tilde f)^2  (\sigma_W^2 / \tilde \sigma^2) \pt{K_X - \sigma_X^2 I_n }+O(  \epsilon_n)\\
   \Sigma =  \E\br{y y^\top} &=  \E\br{\tilde f(u) \tilde f(u)^\top } + \sigma_D^2 I_n \\
    &= \pt{\vertii{\tilde f}^2  -  \frac{\sigma_W^2 \sigma_X^2  }{\tilde \sigma^2} \zeta_1(\tilde f)^2 + \sigma_D^2 } I_n + \zeta_1(\tilde f)^2 \frac{\sigma_W^2}{\tilde \sigma^2}K_X + O(  \epsilon_n)\\
    &= \Sigmal + O(  \epsilon_n).
\end{align*}
The proof for the Stieltjes transforms is similar.
\end{proof}

\subsection{Propagation of the approximate orthogonality}
The contents of this paragraph will not be useful in this section, but rather in later stages of the article to study multi-layers networks by induction. Similar results may be found in \cite[Section D]{FW20} under the name of propagation of approximate orthogonality (see Remark \ref{ANNDetAssrem1} for more details on this concept), and proved by slightly different means.

\begin{lem}\label{PropOrth}
    Let $\sigma_Y^2 = \vertih{\tilde f}^2 + \sigma_D^2  $, and $\Delta_Y = K- \sigma_Y^2 I_n=Y^\top Y / d - \sigma_Y^2 I_n$. Under Assumptions \ref{ANNDetAss}, there exists an event $\cal B$ with $\Pp(\cal B^c) \leq c e^{-n/c} $ for some constant $c>0$, such that uniformly on $\cal B$, $\vertii{\vec{\rm{diag}}(\Delta_Y)}$ and $\vertiii{K}$ are bounded, and $\vertim{\Delta_Y} \leq O\pt{ \vertim{\Delta_X} + \sqrt{\log n / n }}$.
\end{lem}
\begin{proof} We have $\vertiii{\E\br{Y}}^2\leq \vertif{\E\br{Y}}^2 \leq d \vertii{\E[y]}^2\leq O(n)$. From Proposition \ref{ConcProd}, there exists a constant $c>0$ and an event $\cal B$ with $\Pp(\cal B^c) \leq c e^{-n/c} $, such that $\pt{K| \cal B} \propto_{\vertif{\cdot}} \cal E(1/\sqrt n)$, and $\vertiii{Y}\leq 2 c \sqrt n$ on $\cal B$, thus $\vertiii{K}$ is bounded on $\cal B$. We will check the other statements in expectation first, and in a second stage use concentration to obtain bounds on the random objects.

For any $i \in \lint 1,n\rint $, $Y_{ii}$ has the same distribution as $\tilde f(u_i)+\tilde b_i$, where $u_i$ and $\tilde b_i$ are centered Gaussian variables, independent, with variances $1 + O\pt{(\Delta_X)_{ii}}$ and $\sigma_D^2$ respectively. Using  the first order expansion given by Corollary \ref{Expansionzeta} applied to the function $\tilde f^2$, uniformly on $i \in \lint 1,n\rint $ we have:
\begin{align*}
    \E\br{Y_{ii}^2} &= \E\br{\pt{\tilde f(u_i)+b_i}^2} \\
    &= \zeta_0(\tilde f^2) + O\pt{(\Delta_X)_{ii}} + \sigma_D^2 \\
    &= \vertih{\tilde f}^2 + \sigma_D^2 + O\pt{(\Delta_X)_{ii}} \\
     &= \sigma_Y^2 + O\pt{(\Delta_X)_{ii}} .
\end{align*}
We deduce that $\E\br{\vertii{\vec{\rm{diag}}(\Delta_Y)}^2} \leq {\sum_{i=1}^n O\pt{(\Delta_X)_{ii}^2}} \leq O\pt{\vertii{\vec{\rm{diag}}(\Delta_X)}^2} \leq O(1)$. Since $\pt{\vec{\rm{diag}}(\Delta_Y) |\cal B}\propto_{\vertii{\cdot}} \cal E(1/\sqrt n)$, $\vertii{\vec{\rm{diag}}(\Delta_Y)-\E\br{\vec{\rm{diag}}(\Delta_Y)}}\leq O(1)$ a.s. on $\cal B$, and $\vertii{\vec{\rm{diag}}(\Delta_Y)}\leq O(1)$ a.s. on $\cal B$.

Finally $\vertim{K-\Sigma} \leq O\pt{\sqrt{\log n/n}}$ a.s. on $\cal B$ using the general properties of concentration (see Proposition \ref{PropConc}(10)), and from Proposition \ref{SigmaLinNorm}: \begin{align*}\vertim{\Sigma - \sigma_Y^2 I_n }&\leq \vertim{\Sigma-\Sigmal} + O\pt{ \vertim{\Delta_X}} \leq O\pt{ \vertim{\Delta_X} },\end{align*} which proves the result after a final triangular inequality.
\end{proof}

\subsection{Deterministic equivalent and consequences}

We let $\nu^\Sigmal = \mmp(\gamma_n) \boxtimes \mu_\Sigmal$, and we refer to the beginning of Section \ref{Sec4} for the definition of the deterministic equivalent matrix $\bf G_\boxtimes^{\Sigmal}(z)$.

\begin{rem}\label{Linthecaseb=0}  If $\frak b = 0$, then $\Sigmal$ does not depend on $X$ and the objects $\nu^\Sigmal$ and $\bf G_\boxtimes^{\Sigmal}(z)$ are fully explicit. Indeed $\Sigmal = \frak a I_n$, $\mu_\Sigmal = \delta_{\frak a} $, and $\nu^\Sigmal = \mmp(\gamma_n) \boxtimes \delta_{\frak a}  = \frak a \mmp(\gamma_n)$. Since $\Sigmal$ is a multiple of the identity matrix, so are $\cal G_\Sigmal(z)$ and $\bf G_\boxtimes^\Sigmal$, hence:
\begin{align*}
   \bf G_\boxtimes^{\Sigmal}(z) &= \pt{\frac 1 n \Tr  \bf G_\boxtimes^{\Sigmal}(z)} I_n = g_{\frak a \mmp(\gamma_n)}(z) I_n.
\end{align*}
Note that if $f \neq 0$, then $\frak a \neq 0$ and $g_{\frak a \mmp(\gamma_n)}(z)=\frak a {g_{\mmp(\gamma_n)}\pt{  z / \frak a }}$. 

In the case where $\frak  b \neq 0$, we can describe the deterministic equivalents as functions of $X$. Indeed $\mu_\Sigmal = \frak a + \frak b \mu_{K_X}$,  $\nu^\Sigmal = \mmp(\gamma_n) \boxtimes \pt{\frak a + \frak b \mu_{K_X}}$, and since $\cal G_\Sigmal(z) = \frac 1 {\frak b} \cal G_{K_X}\pt{\frac{z-\frak a}{\frak b}}$:
\begin{align*}   
    \bf G_\boxtimes ^{\Sigmal}(z) &=  z^{-1} l_{\check \nu^\Sigmal}(z)  \cal G_L\pt{  l_{\check \nu^\Sigmal}(z)  } \\
    &= \frac{ l_{\check \nu^\Sigmal}(z) }{\frak b z} \cal G_{K_X}\pt{ \frac{ l_{\check \nu^L}(z) - \frak a}{\frak b} } ,
\end{align*}where $\check \nu^\Sigmal = (1-\gamma_n ) \cdot\delta_0 + \gamma_n\cdot  \nu^\Sigmal$ and $l_{\check \nu^\Sigmal}(z) = -1/g_{\check \nu^\Sigmal}(z)$.
\end{rem}

For the next result let us denote:
\begin{align*}
        \epsilon_n &= \frac 1 n +\vertim{\Delta_X}+\sqrt n  \zeta_2(\tilde f)^2 \vertim{\Delta_X}^2 +\sqrt n \zeta_3(\tilde f)^2 \vertim{\Delta_X}^3 , \\  \epsilon'_n &= \frac 1 {\sqrt n} +\vertim{\Delta_X}+n  \zeta_2(\tilde f)^2 \vertim{\Delta_X}^2 +n \zeta_3(\tilde f)^2 \vertim{\Delta_X}^3 .
    \end{align*}
\begin{thm}\label{ANNDetDetEq} Uniformly under Assumptions \ref{ANNDetAss}, the following concentration properties hold true:
\begin{enumerate}
    \item $g_K (z) \propto  \cal E\pt{O_z(1/ n)}$ and  $\cal G_K (z) \propto_{\vertif{\cdot}}   \cal E\pt{O_z(1/\sqrt n)}$.
    \item $g_K(z) \in g_{\nu^L}(z)\pm \cal E\pt{O_z(\epsilon_n)}$ and $\cal G_K (z)\in_{\vertiii{\cdot}}  \bf G_\boxtimes ^{\Sigmal} (z) \pm \cal E\pt{O_z(\epsilon'_n)}$.
\end{enumerate}
\end{thm}

\begin{rem} \label{ANNDetAssrem1} Assumption \ref{ANNDetAss}(4) states roughly that the data matrix is close to being an orthogonal matrix (up to rescaling), a notion that was first introduced in \cite{FW20} under the name of approximate orthogonality. If this is true, the Conjugate Kernel model may be compared to a classical model derived from matrices with i.i.d. entries. To clarify this point, the assumption on $\vertim{\Delta_X}$ measures how the entries of $Y$ are dependent and how far they are from being standard Gaussian random variables individually. The additional bound on $\vertii{\vec{\rm{diag}}(\Delta_X)}$ ensures that the latter phenomenon occurs somewhat uniformly.

    In order for the deterministic equivalents to be meaningful, we need that $\epsilon_n$ converges to $0$ for the Stieltjes transforms, and that $\epsilon'_n$ converges to $0$ for the resolvent matrices. This is a stronger assumption that depends both on the data matrix $X$ and on the activation function $f$. More precisely, $K_X$ should not be too far from $I_n$ entry-wise, to an extent that also depends on how far $f$ is from acting linearly on Gaussians.
    
    In practice, for the Stieltjes transforms and a general activation function $f$ we need $\vertim{\Delta_X}=o\pt{n^{-1/4}}$. If $\zeta_2(\tilde f)=0$, which happens for instance if $f$ is odd symmetric, this convergence can be relaxed to $\vertim{\Delta_X} = o\pt{n^{-1/6}}$. If $\zeta_2(\tilde f)=\zeta_3(\tilde f)=0$, no additional hypothesis is required as  $\epsilon_n = \vertim{\Delta_X}$ converges already to $0$. For the resolvent matrices, the equivalent statements boil down to  $\vertim{\Delta_X}=o\pt{n^{-1/2}}$ in general, $\vertim{\Delta_X} = o\pt{n^{-1/3}}$ if $\zeta_2(\tilde f)=0$, and no additional hypothesis if $\zeta_2(\tilde f)=\zeta_3(\tilde f)=0$.
\end{rem}

\begin{proof}[Proof of Theorem \ref{ANNDetDetEq}] $Y$ satisfies the Assumptions \ref{DetEqAss} as seen in Proposition \ref{propa}. We deduce from Theorem \ref{DetEqSimpl} the Lipschitz concentration properties (1) and the linear concentration properties  $\cal G_K (z)\in_{\vertif{\cdot}}  \bf G_\boxtimes ^{\Sigma}(z)  \pm \cal E\pt{O_z(1/\sqrt n)}$ and $g_K (z)\in g_{\nu}(z)\pm \cal E\pt{O_z(1/n)}$. Moreover from Theorem \ref{72} and Corollaries \ref{propb} and \ref{GboxLip}:
\begin{align*}
   \verti{ g_{\nu^\Sigma}(z)-g_{\nu^\Sigmal}(z)} &\leq O_z\pt{\verti{g_\Sigma(z)-g_\Sigmal(z)}}\\
   &\leq O_z(\epsilon_n),\\
   \vertiii{ \bf G_\boxtimes^{\Sigma} (z)- \bf G_\boxtimes^{\Sigmal}(z)} & \leq O_z\pt{ \vertiii{\Sigma - \Sigmal}} \\&\leq O_z(\epsilon_n'),
\end{align*}which imply the linear concentration properties (2).
\end{proof}

Similarly to Corollary \ref{DetEqCor}, we may deduce from the deterministic equivalents the following spectral properties:
\begin{cor}\label{CorDetAnn}
   Uniformly under Assumptions \ref{DetEqAss}:
   \begin{enumerate}
       \item   $\verti{g_K(z) - g_{\nu^\Sigmal}(z)} \leq \sqrt {\log n} \, O_z(\epsilon_n) $ a.s. and     $\vertim{{\cal G_K(z) - \bf G_\boxtimes^\Sigmal(z) }}  
       \leq \sqrt {\log n} \, O_z(\epsilon_n')$ a.s.
    \item 
 If $f$ is not linear, or if is $f$ linear and the eigenvalues of $K_X$ are bounded from below,  there exists $\theta >0$ such that $D(\mu_K, \nu^\Sigmal) \leq O(\epsilon_n^{\theta})$ a.s. 
\item If moreover $\mu_{K_X}$ converges weakly to a measure $\mu_\infty$ and if $\gamma_n \to \gamma_ \infty$, then $\mu_K$ converges weakly to $\nu_\infty = \mmp(\gamma_\infty) \boxtimes \pt{\frak a + \frak b \mu_\infty}$, and more precisely:  \[ D(\mu_K,\nu_\infty) \leq O\pt{D(\mu_{K_X},\mu_\infty)+|\gamma_n-\gamma_\infty|+\epsilon_n^{\theta}} \quad \rm{a.s.}\]
   \end{enumerate}
\end{cor}

\begin{proof} By definition of the $O_z(\epsilon_n')$ notation, we may find $\alpha>0$ such that uniformly in $z \in \C^+$ with bounded $\Im(z)$, $\cal G_K (z)\in_{\vertiii{\cdot}}  \bf G_\boxtimes ^{\Sigmal} (z) \pm \cal E\pt{O(\epsilon'_n(z))}$ where $\epsilon'_n(z)=\epsilon'_n \frac{|z|^\alpha}{\Im(z)^{2 \alpha}}$. Let us now fix $z \in \C^+$. The maps $M \mapsto M_{ij}$ are linear and $1$-Lipschitz with respect to the spectral norm. By definition of the linear concentration, there are constants $C>0$ such that for any $n,t,i$ and $j$:
\[ \Pp\pt{\verti{\pt{\cal G_K(z)-\bf G_\boxtimes^\Sigmal(z)}_{ij}}\geq t} \leq C e^{- \frac{t^2}{C \epsilon'_n(z)^2} }.\]
We choose $t_n =\epsilon'_n(z)  \sqrt{4 C \log n }$, and we use  a union bound:
\begin{align*}
  \Pp\pt{  \vertim{{\cal G_K(z) - \bf G_\boxtimes^\Sigmal(z) }} \geq t_n} \leq n^2 C e^{- \frac{ t_n^2}{C \epsilon'_n(z)^2} } = C e^{2 \log n - 4 \log n} = C /  n^2.
\end{align*}
These probabilities are summable, so using Borel-Cantelli lemma:
\[ \vertim{{\cal G_K(z) - \bf G_\boxtimes^\Sigmal(z) }}\leq t_n \leq O\pt{ \sqrt {\log n} \, \epsilon'_n(z)} \quad \rm{a.s.} \] We deduce that $\vertim{{\cal G_K(z) - \bf G_\boxtimes^\Sigmal(z) }}  \leq \sqrt {\log n} \, O_z(\epsilon_n')$ a.s. The proof for the Stieltjes transforms is similar.

If $f$ is not linear, or if $f$ is linear and the eigenvalues of $K_X$ are bounded from below, then the eigenvalues of $\Sigmal$ are bounded from below. In this case, for values of $\gamma_n \leq 1$ the cumulative distribution function $\cal F_{\nu^\Sigmal}$ are uniformly Hölder continuous with exponent $\beta = 1/2$, and for values of $\gamma_n \geq 1$ the cumulative distribution function $\cal F_{\check \nu^\Sigmal}$ are uniformly Hölder continuous with exponent $\beta = 1/2$ (see \cite[Section 8.3]{moi1} and Remark \ref{remHoldtrick}). In both cases, from Proposition \ref{OzandKol} we deduce the second assertion.

 Finally for the third assertion, using the properties of free multiplicative convolution \cite[Proposition 4.13]{bercovici1993free} we deduce that $\nu^\Sigmal$ converges weakly to $\nu_\infty$, and that:
\begin{align*}
 D(\nu^\Sigmal,\nu^\infty) &\leq D(\mmp(\gamma_n),\mmp(\gamma_\infty))+D(\mu_\Sigmal,\frak a+ \frak b \mu_\infty) \\
&\leq O\pt{|\gamma_n-\gamma_\infty|+D(\mu_{K_X},\mu_\infty)}. 
\end{align*}
\end{proof}

\subsection{Application to another model involving entry-wise operations} \label{practicalex}

In this paragraph we show how our framework applies to other models of random matrices, not strictly related to artificial neural networks.

We consider $U \in \R^{n \times n}$ a Gaussian random matrix filled with $\cal N$ random variables, i.i.d. within the columns and weakly correlated within the rows. More precisely, we consider $u \in \R^n$ a Gaussian random vector, centered, with covariance matrix:
\[ S = \begin{pmatrix} 1 & 1/n & \dots & & 1/n \\ 1/n & 1 & 1/n && \vdots& \\
\vdots & 1/n & \ddots & & \\
& &  &\ddots & 1/n \\
1/n & \dots & & 1/n & 1\end{pmatrix}.\] We let $U$ be the random matrix with i.i.d. sampled columns from the distribution of $u$. Let $f=\tanh$ be the hyperbolic tangent function, $B$ and $D\in \R^{n \times n}$ independent random matrices filled with i.i.d. $\cal N$ random variables, and $Y = f(U+B)+D$. We want to apply the contents of this section to study the spectral properties of the sample covariance matrix $K= Y^\top Y / n$, its resolvent $\cal G_{K}(z)=\pt{K - z I_n}^{-1}$ and Stieltjes transform $g_K(z) =(1/n)  \Tr \cal G_{K}(z) $.

Let us denote by $J \in \R^{n \times n}$ the matrix whose entries are all equal to $1$, so that $S = I_n + (J-I_n)/n$. 
In law $U=XW$ where $W \in \R^{n \times n}$ is a matrix filled with i.i.d. $\cal N$ entries, independent from the other sources of randomness, and $X = S^{1/2}$. Up to a transposition in the independence structure, the model $Y$ is equal in law as $f(XW+B)+D$, and satisfies the Assumptions \ref{ANNDetAss}, with  $\Delta_X =  (J-I_n)/n$ and $\sigma_W^2=\sigma_X^2=\sigma_B^2=\sigma_D^2=1$. Indeed $\tilde \sigma^2 = 2$, $
    \tilde f(t) = \tanh(\sqrt 2 t)$, and by odd symmetry of $\tanh$ it is clear that $  \zeta_0(\tilde f)=\zeta_2(\tilde f)=0$. We also have $\vertiii{\Delta_X}=1-1/n$, $\vec{\rm{diag}}(\Delta_X)=0$, $\vertim{\Delta_X}=1/n$ and $\epsilon_n = O(1/n)$.
    
    The constants $\frak a$ and $\frak b$ are linked to Gaussian moments of the hyperbolic tangent function and may be numerically approximated. 
    The matrix $\Sigmal = \frak a I_p + \frak b  S = (\frak a + \frak b) I_p + \frak b (J-I_n) / n$ is explicitly diagonalizable and $\mu_\Sigmal =  \frac {n-1} n \cdot \delta_{\frak a + \frak b-\frak b / n} + \frac 1 n \cdot \delta_{\frak a + 2 \frak b -\frak b / n }$. The measure $\nu^\Sigmal  = \tilde \nu^\Sigmal = \mmp(1) \boxtimes \mu_L$ is characterized by its Stieltjes transform $g_{\nu^\Sigmal}(z)$, which is the only solution $g \in \C^+$ of the self-consistent equation:\begin{align*} g &= \int_\R \frac 1 {-zg t- z} \mu_\Sigmal(dt) \\
&= \frac 1 {-zg (\frak a + \frak b-\frak b / n)- z} + \frac 1 n \pt{  \frac 1 {-zg (\frak a + 2\frak b-\frak b / n)- z} - \frac 1 {-zg (\frak a + \frak b-\frak b / n)- z}}\end{align*}
This equation may be rewritten as a cubic polynomial equation and solved explicitly. The deterministic equivalent matrix $\bf G_\boxtimes^\Sigmal(z)$ may also be explicitly computed. The interested reader can check that:
\begin{align*}
     \bf G_\boxtimes^\Sigmal(z) %&= \pt{- z g_{\nu_n^L}(z) L - z I_n}^{-1} \\
   %  &= \pt{\pt{- z g_{\nu_n^L}(z) \frak b / n } J + \pt{- z g_{\nu_n^L}(z) (\frak a + \frak b-\frak b / n) -z}I_n}^{-1} \\
     &= - \frac {  g_{\nu^\Sigmal}(z) \frak b  } {\pt{  g_{\nu^\Sigmal}(z) (\frak a + \frak b-\frak b / n) +1}\pt{  { g_{\nu^\Sigmal}(z) (\frak a + 2 \frak b-\frak b / n) +1}}} \frac J n \\
     &\quad      - \frac 1 {\pt{  g_{\nu^\Sigmal}(z) (\frak a + \frak b-\frak b / n) +1}} \frac {I_n} z 
\end{align*}

Theorem \ref{ANNDetDetEq} applied to this model provides the following deterministic equivalents: $g_K(z) \in g_{\nu^\Sigmal}(z)\pm \cal E\pt{O_z(1/n)}$, and $\cal G_K (z)\in_{\vertiii{\cdot}}  \bf G_\boxtimes ^{\Sigmal} (z) \pm \cal E\pt{O_z(1/\sqrt n)}$.

Since $\mu_\Sigmal$ converges weakly to $\delta_{\frak a + \frak b}$, $\nu^L$ and $\mu_K$ converge weakly to $\nu_\infty = \mmp(1) \boxtimes \delta_{\frak a + \frak b} = \pt{\frak a + \frak b} \mmp(1)$, and $D(\mu_K,\nu^\Sigmal)\leq O(n^{-\theta})$ for some $\theta>0$. To go further, Corollary \ref{CorDetAnn}(3) is not helpful since the measures are  discrete. As a matter of fact, $D\pt{\mu_\Sigmal,\delta_{\frak a + \frak b} }=1-1/n$ does not vanish. However we can compare directly the Stieltjes transforms of $\Sigmal$ and $\delta_{\frak a+ \frak b}$:
\begin{align*}
    \verti{g_{\Sigmal}(z)-g_{\delta_{\frak a+ \frak b}} (z)}&=\verti{ \frac{1-1/n} {\frak a + \frak b - \frak b/n - z}+ \frac {1/n} {\frak a + 2\frak b - \frak b/n - z} - \frac{1}{\frak a + \frak b - z} } \\
    &\leq O_z(1/n)
\end{align*}
From Theorem \ref{72}, we have $\verti{g_{\nu^\Sigmal(z)}-g_{\nu^\infty}(z)}\leq O_z(1/n)$, hence $D(\nu^\Sigmal,\nu^\infty)\leq O(n^{-\theta})$ for some $\theta>0$ from Proposition \ref{OzandKol}. We conclude that $\mu_K$ converges to $\nu_\infty$ as speed $O(n^{-\theta})$ in Kolmogorov distance.

\newpage

\section{Single-layer neural network with random data} \label{Sec6}

\subsection{Setting} In this section we study the Conjugate Kernel matrix associated to a single-layer artificial neural network with random input. The hypothesis are thus the same as in the last section, excepted for the random data matrix. We consider:
\begin{itemize}
     \item a random weight matrix $W \in \R^{d \times d_0}$, with variance parameter $\sigma_W^2>0$,
    \item a random data matrix $X \in \R^{ d_0 \times n}$, $\sigma_X^2 > 0$,
    \item two random biases matrices $B$ and $D \in \R^{d \times n }$, $\sigma_B^2,\sigma_D^2 \geq 0$,
    \item and an activation function $f : \R \to \R$.
\end{itemize} As output of the neuron, we set $Y =f(WX/\sqrt{d_0}+B)+D \in \R^{d \times n}$, where the function $f$ is applied entry-wise. Our goal is to investigate the spectral properties of the Conjugate Kernel matrix $K= Y^\top Y / d$. For $z \in \C^+$ we define its resolvent $\cal G_{K}(z)=\pt{K - z I_n}^{-1}$ and Stieltjes transform $g_K(z) =(1/n)  \Tr \cal G_{K}(z) $. We also define the following objects:
\begin{align*}
    \tilde \sigma^2 &= \sigma_W^2 \sigma_X^2 + \sigma_B^2 , \\
    \tilde f (t)&= f(\tilde \sigma t), \\ 
    \frak a &=  {\vertih{\tilde f}^2  -  \frac{\sigma_W^2\sigma_X^2}{ \tilde \sigma^2} \zeta_1(\tilde f)^2 + \sigma_D^2 }, \\
    \frak b &= \zeta_1(\tilde f)^2 \frac{\sigma_W^2}{\tilde \sigma^2},\\K_X &= X^\top X / {d_0}, \\ \Delta_X &= K_X- \sigma_X^2 I_n, \\
 \Sigma_X &= \frak a I_n +  \frak b K_X.
\end{align*}

Like in the rest of this article, we sometimes omit the indices $n$ and $z$ for a better readability, even if we are implicitly dealing with sequences of matrices, measures, and complex functions. For reasons that we will understand later, we do not make hypothesis on the random matrix $X$ itself, but rather on $X$ conditioned with a high probability event (see the definition before Proposition \ref{ConcProd}). 

\begin{ass} \label{ANNRndAss} 
\begin{enumerate} \item $W$, $B$ and $D$ are random, independent, with i.i.d. $\cal N(\sigma_W^2)$, $\cal N(\sigma_B^2)$ and $\cal N(\sigma_D^2)$ entries respectively.    \item $\tilde f$ is Lipschitz continuous and Gaussian centered, that is $\E\br{\tilde f(\cal N)}=\E\br{f(\tilde \sigma \cal N)}=0$.
 \item $X$ is random, independent from $W$, $B$ and $D$.  There is an event $\cal B$ with $\Pp\pt{\cal B^c} \leq O(\sqrt{\log n}/n)$, such that  $ (X|\cal B) \propto_{\vertif{\cdot} }\cal E(1)$.
 \item There is a sequence $\epsilon_n$ converging to $0$, such that uniformly in $\omega \in \cal B$, $\vertiii{K_X}$ and $\vertii{\vec{\rm{diag}}(\Delta_X)}$ are bounded, and $\vertim{\Delta_X} \leq O(\epsilon_n)$.
     \item The ratio $\gamma_n= \frac  n d$ is bounded from above and away from $0$.
     \item There is a sequence $\hat \epsilon_n\geq 0$ such that $\verti{\E\br{g_{K_X}(z)} - g_\tau (z)} \leq O_z(\hat\epsilon_n)$ for some sequence of measures $\tau$ supported on $\R^+$.
    \item There is a sequence $\hat \epsilon'_n\geq \hat \epsilon_n$ such that $\vertiii{\E\br{\cal G_{K_X}(z)} - \bf H (z)} \leq O_z(\hat\epsilon'_n)$ for some sequence of matrix functions $\bf H : \C^+ \to \C^{n \times n}$ satisfying $\vertiii{\bf H}\leq 1 /\Im(z)$.
\end{enumerate}
\end{ass}

\begin{rem}\label{hpeyolo} Compared to Assumptions  \ref{ANNDetAss}, in the assertions (3) and (4) we ask the data matrix $X$ to be independent from the other random matrices, well concentrated, and  uniformly approximately orthogonal on an event of high probability.

Let us explain why $\Pp\pt{\cal B^c} \leq O(\sqrt{\log n}/n)$ is a convenient choice to simplify our results. We claim that, up to a $O_z(\sqrt{\log n}/n)$ error that will blend into similar terms, it will be equivalent for us to compute expectations on the whole probability set or on $\cal B$. Indeed, if a random function $\zeta : \N \times \C^+ \to \C$ satisfies the \emph{a priori} bound $\zeta(n,z) \leq 1/\Im(z)$, then we have:
\begin{align*}
    \verti{\E_{\cal B}\br{\zeta } -\E\br{\zeta}}&= \verti{\E\br{\bf 1_{\cal B}\zeta} / \Pp(\cal B) -\E\br{\zeta}} \\
    &= \frac {\verti{\E\br{\bf 1_{\cal B^c}\zeta} + \E\br{\zeta}\Pp(\cal B^c)} } {\Pp(\cal B)}\\
    &\leq \frac{2\Pp\pt{\cal B^c}}{1-\Pp\pt{\cal B^c}} \frac 1 {\Im(z)} \\ &\leq  O_z(\Pp\pt{\cal B^c}) \leq O_z(\sqrt{\log n}/n).
\end{align*}

The assertions (6) and (7) correspond   to deterministic equivalents for the Stieltjes transform and the resolvent of $K_X$. If $\zeta_1(\tilde f) = 0$ then $\frak b =0$, and as seen in Remark \ref{Linthecaseb=0} the deterministic equivalents will not depend on $X$. In this case $\tau$ and $\bf H(z)$ will not appear in the results and the assertions (6) and (7) are in essence  empty.
\end{rem}

 \subsection{Deterministic equivalent and consequences}
 
We let $\nu^\Sigmax = \mmp(\gamma_n) \boxtimes \mu_\Sigmax$, and we refer to the beginning of Section \ref{Sec4} for the definition of the deterministic equivalent matrix $\bf G_\boxtimes^{\Sigmax}(z)$.
Similarly to the process used to express $\bf G_\boxtimes^\Sigmax(z)$ as a function of $\mu_{K_X}$ and $\cal G_{K_X}$ (see Remark \ref{Linthecaseb=0}), we define the objects:
\begin{align*}
    \chi &= \mmp\pt{\gamma_n } \boxtimes \pt{\frak a + \frak b \tau},\\
    \check \chi  &= (1-\gamma_n ) \cdot\delta_0 + \gamma_n \cdot \chi, \\
    l_{\check \chi}(z) &= -1/g_{\check \chi}(z), \\
     \bf K(z) &=\frac{ z^{-1} l_{\check \chi}(z) }{\frak b} \bf H\pt{\frac{l_{\check \chi}(z) - \frak a}{\frak b} } &\rm{if} \, \frak b \neq 0 ,\\
     \bf K(z) &=g_{\frak a \mmp(\gamma_n)}(z) I_n &\rm{if} \, \frak b = 0.
\end{align*}

\begin{lem} \label{LemEbfG}Under Assumptions \ref{ANNRndAss}, $\verti{\E\br{g_{\nu^\Sigmax}(z)}-g_\chi(z)} \leq \zeta_1(\tilde f) O_z(\hat\epsilon_n+ \sqrt{\log n}/n)$, and 
    $\vertiii{\E\br{\bf G_\boxtimes^  \Sigmax(z)} - \bf K(z) } \leq \zeta_1(\tilde f) O_z(\hat\epsilon_n'+1/\sqrt n) $.
\end{lem}

\begin{proof}
    If $\zeta_1(\tilde f) = 0$, then $\Sigmax$ does not depend on $X$, $\bf G_\boxtimes^{\Sigmax}(z) = \bf K(z)$ and $\nu^\Sigmax = \chi$. If $\zeta_1(\tilde f) \neq 0$, from Proposition \ref{ConcProd} applied to $(X|\cal B)$ there is a constant $c>0$ and an event $\cal B' \subset \cal B$ with $\Pp\pt{\cal B'^c} \leq \Pp\pt{\cal B^c}+ce^{-n/c}\leq O(\sqrt{\log n}/n)$, on which $( \Sigmax | \cal B') \propto_{\vertif{\cdot}} \cal E\pt{1/\sqrt n}$. 
    As explained in the Remark \ref{hpeyolo}, given the \emph{a priori} bounds on Stieltjes transforms and on the spectral norm of resolvent matrices, we may
 pass from expectations on $\cal B'$ to expectations on the full probability set at the cost of a $O_z(\sqrt{\log n} / n) $ error term. We thus have:
    \begin{align*}
        \verti{\E_{\cal B'}\br{g_\Sigmax(z)}-g_{\frak a+\frak b \tau}(z)} &=  \verti{\E\br{\bf 1_{\cal B'}\pt{g_\Sigmax(z)-g_{\frak a+\frak b \tau}(z)}}}\, / \, \Pp(\cal B') \\  
        &\leq \verti{\frac 1 {\frak b} \E\br{g_{K_X}\pt{\frac{z-\frak a}{\frak b}}}-\frac 1 {\frak b} g_\tau\pt{\frac{z-\frak a}{\frak b}}}  +    O_z(\sqrt{\log n} / n)  \\     
        &\leq O_z(\hat\epsilon_n+\sqrt{\log n}/n) .\end{align*} 
Proposition \ref{ApproxbgGRnd} applied to $(\Sigma_X|\cal B')$ implies that:
\[\verti{\E_{\cal B'}\br{g_{\nu^\Sigmax}(z)}-g_\chi(z)}\leq  O_z(\hat\epsilon_n+ \sqrt{\log n}/n),\] hence $\verti{\E\br{g_{\nu^\Sigmax}(z)}-g_\chi(z)} \leq  O_z(\hat\epsilon_n+ \sqrt{\log n}/n)$. Similarly for the resolvents:\begin{align*}
        \vertiii{\E_{\cal B'}\br{\cal G_\Sigmax(z)} - \frac 1 {\frak b }\bf H\pt{\frac{z-\frak a}{\frak b}} }&=\vertiii{\E\br{ \frac 1 {\frak b }\cal G_{K_X}\pt{\frac{z-\frak a}{\frak b}}} - \frac 1 {\frak b }\bf H\pt{\frac{z-\frak a}{\frak b}} } \\
        &\quad+     O_z(\sqrt{\log n} / n)  \\
        &\leq  O_z(\hat\epsilon'_n+\sqrt{\log n} / n) ,\end{align*} which implies by Proposition \ref{ApproxbgGRnd} that  $\vertiii{\E_{\cal B'}\br{\bf G_\boxtimes^\Sigmax (z)}- \bf K(z)}\leq O_z(\hat\epsilon_n+\hat\epsilon'_n+1/\sqrt n)\leq O_z(\hat\epsilon'_n+1/\sqrt n)$, and $\vertiii{\E\br{\bf G_\boxtimes^\Sigmax (z)}- \bf K(z)}\leq O_z(\hat\epsilon'_n+1/\sqrt n)$.\end{proof}

        We remind our reader that the sequences $\epsilon$ appearing in the Assumptions \ref{ANNRndAss} measure the lack of orthogonality of $X$ for $\epsilon_n$, and the convergence speeds in the deterministic equivalents $g_{K_X}(z)\approx g_\tau(z)$ for $\hat\epsilon_n$  and $\cal G_{K_X}(z)\approx \bf H(z)$ for $\hat\epsilon'_n$ respectively. 
For the next result let us denote:
\begin{align*}
      \tilde   \epsilon_n &= \sqrt{\log n} /  n + \zeta_1(\tilde f) \hat \epsilon_n +\epsilon_n + \sqrt n \zeta_2(\tilde f)^2 \epsilon_n^2 + \sqrt n \zeta_3(\tilde f)^2 \epsilon_n^3 , \\
      \tilde   \epsilon_n' &= 1/ \sqrt{n} +  \zeta_1(\tilde f) \hat \epsilon_n' +  \epsilon_n +  n \zeta_2(\tilde f)^2 \epsilon_n^2 +  n \zeta_3(\tilde f)^2 \epsilon_n^3.
    \end{align*}
    $\tilde \epsilon_n$ and  $\tilde \epsilon'_n$ will correspond to  the new convergence speeds in the deterministic equivalents $g_{K}(z)\approx g_\chi(z)$  and $\cal G_{K}(z)\approx \bf K(z)$.

     \begin{thm}\label{ANNRndDetEq} Uniformly under Assumptions \ref{ANNRndAss}, there is an event $\cal B' \subset \cal B$ with $\Pp(\cal B'^c) \leq O(\sqrt {\log n}/n)$, such that the following conditional expectation properties hold true:
     \begin{enumerate}
         \item $ (Y | \cal B' )\propto_{\vertif{\cdot} }\cal E(1)$, $ ( g_{K}(z)  | \cal B' )  \propto \cal E(O_z(1/n))$, and $ (\cal G_{K} (z)| \cal B' )  \propto_{\vertif{\cdot} }\cal E(O_z(1/\sqrt n)) $.
         \item $( g_{K} (z) | \cal B' ) \in g_{\chi}(z)\pm \cal E\pt{ O_z(\tilde \epsilon_n) }$ and  $(\cal G_{K} (z)| \cal B' )\in_{\vertiii{\cdot}}  \bf K(z) \pm \cal E\pt{ O_z(\tilde \epsilon_n') }  $.
     \end{enumerate}
           \end{thm}
\begin{proof} For a better readability in the upcoming arguments, we choose to omit the spectral parameters $z$ in most of our notations. From Proposition \ref{ConcProd} applied to $(W,(X|\cal B)) \propto_{\vertif{\cdot}} \cal E(1)$  there is a constant $c>0$ and an event $\cal B' \subset \cal B$, with $\Pp\pt{\cal B'^c} \leq \Pp(\cal B^c)+ce^{-n/c}\leq  O(\sqrt {\log n}/n)$, such that $( WX | \cal B') \propto_{\vertif{\cdot}} \cal E\pt{\sqrt n}$. The map $(U,B,D)\mapsto f(U+B)+D$ is Lipschitz with respect to the Frobenius norm, and by independence $\pt{(WX/\sqrt {d_0}\,|\cal B'),B,D}\propto_{\vertif{\cdot}} \cal E\pt{1}$, thus $(Y|\cal B')=f\pt{(WX/\sqrt {d_0}\,|\cal B') + B} +D \propto_{\vertif{\cdot}} \cal E\pt{1} $. The general concentration properties recalled in Proposition \ref{PropConc} imply that  $ ( g_{K}  | \cal B' )  \propto \cal E(O_z(1/n)) $ and $(\cal G_{K} | \cal B' )  \propto_{\vertif{\cdot} }\cal E(O_z(1/\sqrt n) )$.

For the second assertion, given the linear concentration properties $(\cal G_{K}| \cal B' )  \in_{\vertiii{\cdot} } \E\br{(\cal G_{K}   |\cal B' )  }\pm \cal E(O_z(1/\sqrt n)) $ and $ ( g_{K}  | \cal B' )  \in \E\br{ ( g_{K}  | \cal B' )  }\pm\cal E(O_z(1/n)) $, we only need to prove that  $\verti{\E\br{( g_{K} | \cal B' )}-g_\chi} \leq O_z(\tilde \epsilon_n)$, and that $\vertiii{\E\br{(\cal G_{K}  | \cal B' ) }- \bf K}\leq O_z(\tilde \epsilon_n')$. To compare these expectations, as explained in the Remark \ref{hpeyolo},  we may assume without loss of generality that $\cal B' = \Omega$, because the additional $O_z(\sqrt{\log n}/n)$ error terms will not change the final estimates.

  We apply Theorem \ref{ANNDetDetEq} to the models with  deterministic data matrices $X(\omega)$, uniformly in the outcomes $\omega \in \Omega_X$. Since $1/\sqrt n \leq \tilde \epsilon_n'$ and $\epsilon_n +  n \zeta_2(\tilde f)^2 \epsilon_n^2 +  n \zeta_3(\tilde f)^2 \epsilon_n^3 \leq \tilde \epsilon_n'$, we obtain that $\cal G_{K(\omega)} \in_{\vertiii{\cdot}} \bf G_\boxtimes^{\Sigmax(\omega)} \pm \cal E\pt{O_z(\tilde \epsilon_n'}$. As a consequence, we get uniformly in $\omega \in \Omega_X$: 
    \[ \vertiii{\E_{W,B,D}\br{\cal G_K(\omega)}-\bf G_\boxtimes^{\Sigmax(\omega)}}  \leq O_z(\tilde \epsilon_n').\]
    
  Since $X$ is independent from the other sources of randomness, for any measurable function $\Phi$ we have $\E\br{\Phi(W,X,B,D)|X}=\E_{W,B,D}\br{\Phi(W,X,B,D)}$.     We can thus integrate on $X$ the above inequality  and using the tower property of conditional expectation:
    \begin{align*}
        \vertiii{\E\br{\cal G_K}-\E\br{\bf G_\boxtimes^{\Sigmax}}} &= \vertiii{\E\br{\E\br{\cal G_K-\bf G_\boxtimes^{\Sigmax}|X}}}\\
        &\leq \E_X\br{\vertiii{\E_{W,B,D}\br{\cal G_{K} }-\bf G_\boxtimes^{\Sigmax}}}  \leq O_z(\tilde \epsilon_n' ).
    \end{align*}

From Lemma \ref{LemEbfG} we also have $\vertiii{\E\br{\bf G_\boxtimes^\Sigmax }- \bf K}\leq \zeta_1(\tilde f) O_z(\hat  \epsilon_n'+1/\sqrt n)\leq O_z(\tilde \epsilon_n')$, hence $\vertiii{\E\br{\cal G_K}- \bf K}\leq O_z(\tilde \epsilon_n')$. The proof for the Stieltjes transforms is similar.
\end{proof}

\begin{cor}\label{CorRndAnn}
   Uniformly under Assumptions \ref{ANNRndAss}:
   \begin{enumerate}
       \item $\verti{g_K(z) - g_{\chi}(z)} \leq \sqrt {\log n} \, O_z(\tilde \epsilon_n) $ a.s., and $  \vertim{{\cal G_K(z) - \bf K (z) }}  \leq \sqrt {\log n} \, O_z(\tilde \epsilon'_n)$ a.s.
    \item 
 If $f$ is not linear, or if is $f$ linear and the measures $\tau$ are supported on the same compact of $(0,\infty)$,  there exists $\theta >0$ such that $D(\mu_K, \chi) \leq O(\tilde \epsilon_n^{\theta})$ a.s. 
\item If moreover $\tau$ converges weakly to a measure $\tau_\infty$, and if $\gamma_n \to \gamma_ \infty$, then $\mu_K$ converges weakly to $\chi_\infty = \mmp(\gamma_\infty) \boxtimes \pt{\frak a + \frak b \tau_\infty}$, and more precisely:      \[ D(\mu_K, \chi_\infty) \leq O\pt{D(\tau,\tau_\infty)+|\gamma_n-\gamma_\infty|+\tilde \epsilon_n^{\theta}}\quad \rm{a.s.}\]
   \end{enumerate}
\end{cor}

 As we did for Corollary \ref{DetEqCor}, we  will not prove this result here, but rather prompt our reader to consult the proof of Corollary \ref{CorDetAnn} which is extremely similar.

\subsection{Application to data matrices with i.i.d. columns} In this paragraph we focus on a fairly general setting where the data matrix $X$ is made of independent samples,  and we explore the consequences given by our deterministic equivalents. Let us first mention a general framework on which the Assumption \ref{ANNRndAss}(4) holds true.

\begin{prop}[\cite{FW20}, Proposition 3.3] \label{iidsample} Let $X \in \R^{d_0 \times n}$ be a random matrix whose columns  are i.i.d. sampled from the distribution of a random vector $x \in \R^{d_0}$, such that  $x \propto_{\vertii{\cdot}} \cal E(1)$, $\E[x]=0$, and $\E\br{\vertii{x}^2}=\sigma_x^2 d_0$. We also assume the ratio $\frac n {d_0}$ to be bounded from above and away from $0$.

Then there is an event $\cal B$ with $\Pp(\cal B^c)\leq O(1/n)$, such that uniformly in $\omega \in \cal B$, $\vertiii{K_X}$ and $\vertii{\vec{\rm{diag}}(\Delta_X)}$ are bounded, and $\vertim{\Delta_X} \leq O(\epsilon_n)$ with $\epsilon_n = \sqrt{\log n/n}$.
\end{prop}

\begin{rem}     In \cite{FW20} the result is stated for a weaker notion called called convex concentration. We will not digress on this here but rather refer our reader to \cite[Section 1.7]{LC20} which presents in details this variant of concentration.  Also note that in the above result we only need concentration for the columns of $X$, while in our deterministic equivalent result we need concentration for the whole matrix $X$.
\end{rem}

To obtain deterministic equivalents for the input data matrix, it is of course possible to use again Theorem  \ref{DetEqSimpl}:

\begin{prop}
    If $X \in \R^{d_0 \times n}$ is a random matrix, $\propto_{\vertif{\cdot}} \cal E(1)$ concentrated, whose columns  are i.i.d. sampled from the distribution of a random vector $x \in \R^{d_0}$, with $\vertii{\E[x]}$ and $\vertiii{\E[K_X]}$ bounded, and if the ratio $n/d_0$ is bounded from above and away from $0$, then with $\tau = \mmp(n/d_0) \boxtimes \mu_{\E[K_X]}$:
    \begin{align*}
        \verti{\E\br{g_{K_X}(z)} - g_\tau (z)} &\leq O_z(1/n) \\
        \vertiii{\E\br{\cal G_{K_X}(z)} - \bf G_\boxtimes^{\E[K_X]} (z)} &\leq O_z(1/\sqrt n).
    \end{align*}
\end{prop}

If we combine the previous propositions, we obtain deterministic equivalents for the Conjugate Kernel model in a fairly general setting where the data matrix is made of i.i.d. training samples. This encompasses in particular the case where $X$ is a matrix with i.i.d. $\cal N$ entries, which was the original model studied in \cite{PW17}.

\begin{rem}    
Note that the typical order of magnitude $\epsilon_n = \sqrt{\log n/n}$ given by Proposition \ref{iidsample} is good enough for a meaningful equivalent of the Stieltjes transform, with a $O_z(\log n/\sqrt n)$ convergence speed. However it is not small enough for the resolvent, where we would have an $O_z(\zeta_2(\tilde f)^2 \log n)$ error term, unless of course $\zeta_2(\tilde f)=0$ (see Remark \ref{ANNDetAssrem1}). If $\zeta_2(\tilde f)=0$, we obtain a  $O_z\pt{(\log n)^{3/2}/\sqrt n}$ error term for the deterministic equivalent of the resolvent. 

This condition $\zeta_2(\tilde f)=0$, and more generally the Hermite coefficients of the activation function $f$, appear in other articles that study the Conjugate Kernel model. Let us consider indeed $\tilde Y = \sqrt{\frak a} Z + \sqrt{\frak b } WX/\sqrt{d_0}$, and $\tilde K = \tilde Y^\top \tilde Y / d$, where $Z$ is a third random matrix, independent from the others, and filled with i.i.d. $\cal N$ entries. Using Theorem \ref{ANNDetDetEq} conditionally on $X$ and similar arguments to those of this paper, we see that $\cal G_{\tilde K}(z)$ also admits $\bf G_\boxtimes^{\E[K_X]} (z)$ as deterministic equivalent, with a $O_z\pt{(\log n)^{3/2}/\sqrt n}$ error term in the case where $\zeta_2(\tilde f) = 0$. In \cite[Theorem 2.3]{BP19}, using combinatorics it is shown that the biggest eigenvalues of both models behave similarly if $\zeta_2(\tilde f) = 0$. Although we could not manage to retrieve this property with our deterministic equivalent solely, there is without a doubt a connection between these statements.  \cite{BP19} also provides other equivalent models in the case where $\zeta_2(\tilde f) \neq 0$, which we could not relate to our results.
\end{rem}

\newpage

\section{Multi-layer neural network model} \label{Sec7}

In this section we consider the Conjugate Kernel matrix associated to an artificial neural network with $L$ hidden layers and a random input matrix:

\begin{align*}
X_0 &\to X_1 = f_1\pt{W_1X_0/\sqrt{d_0}+B_1}+D_1 \\X_1 &\to X_2 = f_2\pt{W_2X_1/\sqrt{d_1}+B_2}+D_2 \\&\vdots \\
X_l & \to X_{l+1} = f_{l+1} \pt{W_{l+1}X_l/\sqrt{d_l}+B_{l+1}}+D_{l+1} \\
&\vdots \\
X_{L-1} &\to X_{L} = f_{L} \pt{W_{L}X_{L-1}/\sqrt{d_{L-1}}+B_{L}}+D_{L}\end{align*}
The initial data $X_0 \in \R^{d_0 \times n}$ is a random matrix with variance parameter $\sigma_{X_0}^2>0$. Each layer $l \in \lint 1, L \rint$ is made of:
\begin{itemize}
    \item a random weight matrix $W_l \in \R^{d_{l} \times d_{l-1}}$, with variance parameter $\sigma_{W_l}^2>0$,
    \item two random biases matrices $B_l$ and $D_l \in \R^{d_l \times n}$, $\sigma_{B_l}^2,\sigma_{D_l}^2 \geq 0$,
    \item and an activation function $f_l : \R \to \R$.
\end{itemize} At each layer, we define the conjugate kernel matrix $K_l= X_l^\top X_l/d_l$, and for $z \in \C^+$ its resolvent $\cal G_l(z)=\pt{K_l - z I_n}^{-1}$, and Stieltjes transform $g_l(z) = (1/n) \Tr \cal G_l(z)$. We define by induction the following objects:
\begin{align*}
    \tilde \sigma_l^2 &= \sigma_{W_l}^2 \sigma_{X_{l-1}}^2 + \sigma_{B_l}^2 , \\
    \tilde f_l (t)&= f_l(\tilde \sigma_l t), \\ 
     \sigma _{X_l} &= \vertih{\tilde f_l }^2+\sigma_{D_l}^2 ,\\
    \frak a_l &=  {\vertih{\tilde f_l}^2  -  \frac{\sigma_{W_l}^2\sigma_{X_l}^2}{ \tilde \sigma_l^2} \zeta_1(\tilde f_l)^2 + \sigma_{D_l}^2 }, \\
    \frak b_l &= \zeta_1(\tilde f_l)^2 \frac{\sigma_{W_l}^2}{\tilde \sigma_l^2},\\
    \Delta_{X_l} &= K_l- \sigma_{X_l}^2 I_n, \\
 \Sigma_{X_l} &= \frak a_l I_n +  \frak b_l K_X.
     \end{align*}

\begin{ass} \label{ANNMulAss} 
\begin{enumerate} \item $W_l$, $B_l$ and $D_l$ are random, independent as a family for $l \in \lint 1,L\rint$, with i.i.d. $\cal N(\sigma_{W_l}^2)$, $\cal N(\sigma_{B_l}^2)$ and $\cal N(\sigma_{D_l}^2)$ entries respectively.     \item $\tilde f_l$ are Lipschitz continuous and Gaussian centered, that is $\E\br{\tilde f_l(\cal N)}=\E\br{f(\tilde \sigma_l \cal N)}=0$. 
 \item $X_0$ is random, independent from all the other matrices.  There is an event $\cal B$ with $\Pp\pt{\cal B^c} \leq O(\sqrt{\log n}/n)$, such that  $ (X_0|\cal B) \propto_{\vertif{\cdot} }\cal E(1)$.
 \item There is a sequence $\epsilon_n $ converging to $0$, with $\sqrt{\log n / n}\leq O(\epsilon_n)$, such that uniformly in $\omega \in \cal B$, $\vertiii{K_0}$ and $\vertii{\vec{\rm{diag}}(\Delta_{X_0})}$ are bounded, and $\vertim{\Delta_{X_0}} \leq O(\epsilon_n)$.
     \item The ratios $\gamma_n^{(l)}= \frac n {d_l} $ are bounded from above and away from $0$.
    \item There is a sequence $\hat \epsilon_n^{(0)}\geq 0$ such that $\verti{\E\br{g_{K_X}(z)} - g_{\chi_n^{(0)}} (z)} \leq O_z(\hat\epsilon_n^{(0)})$ for some sequence of measures $\chi_n^{(0)}$ supported on $\R^+$.
    \item There is a sequence ${\hat {\epsilon}{}'_n}{}^{(0)}\geq \hat \epsilon_n^{(0)}$ such that $\vertiii{\E\br{\cal G_{K_X}(z)} - \bf G_0 (z)} \leq O_z({\hat {\epsilon}{}'_n}{}^{(0)})$ for some sequence of matrix functions $\bf G_0 : \C^+ \times \C^{n \times n}$ satisfying $\vertiii{\bf G_0(z)}\leq 1 /\Im(z)$. 
\end{enumerate}
\end{ass}

 Starting from the deterministic equivalents $\mu_{K_0} \approx \chi_n^{(0)}$ and $\cal G_{K_0}(z) \approx \bf G_0(z)$, we define by induction for $l \in \lint1, L \rint $:\begin{align*}      
   \chi_n^{(l)} &= \mmp\pt{\gamma_n^{(l)} } \boxtimes \pt{\frak a_l + \frak b_l \chi_n^{(l-1)}} ,\\
    \check \chi_n^{(l)}  &= (1-\gamma_n^{(l)} ) \cdot\delta_0 + \gamma_n^{(l)} \cdot\chi_n^{(l)}, \\
    l_{\check \chi_n^{(l)}}(z) &= -1/g_{\check \chi_n^{(l)}}(z), \\
     \bf G_l(z) &=\frac{ z^{-1} l_{\check \chi_n^{(l)}}(z) }{\frak b_l} \bf G_{l-1}\pt{\frac{l_{\check \chi_n^{(l)}}(z) - \frak a_l}{\frak b_l} } &\rm{if} \, \zeta_1(\tilde f_l) \neq 0 ,\\
     \bf K_l(z) &=g_{\frak a_l \mmp(\gamma_n^{(l)})}(z) I_p &\rm{if} \,  \zeta_1(\tilde f_l) = 0.
\end{align*}

We remind our reader that the sequence $\epsilon_n$ measures the lack of orthogonality of the data matrix $X_0$. We define by induction the following sequences, corresponding respectively to the error terms in the approximation of the Stieltjes transforms and the resolvent at layer $l$:\begin{align*}\hat \epsilon_n^{(l+1)}  &= \sqrt{\log n} /  n + \zeta_1(\tilde f_l) \hat \epsilon_n^{(l)} +\epsilon_n + \sqrt n \zeta_2(\tilde f_l)^2 \epsilon_n^2 + \sqrt n \zeta_3(\tilde f_l)^2 O(\epsilon_n)^3 \\
   \hat \epsilon_n'^{(l+1)}  &= 1 / \sqrt  n +  \zeta_1(\tilde f_l) \hat \epsilon_n'^{(l)}  +\epsilon_n + \sqrt n \zeta_2(\tilde f_l)^2 \epsilon_n^2 + \sqrt n \zeta_3(\tilde f_l)^2 \epsilon_n^3 .\end{align*}

     \begin{thm}\label{ANNMulDetEq} Uniformly under Assumptions \ref{ANNMulAss}, there is a constant $c>0$ and an event $\cal B' \subset \cal B$ with $\Pp(\cal B'^c) \leq O(\sqrt {\log n}/n)$, such that for any $l \in \lint 1, L \rint$, the following conditional expectation properties hold true:\begin{enumerate}
         \item $ (X_l | \cal B' )\propto_{\vertif{\cdot} }\cal E(1)$, $ ( g_{K_l}(z)  | \cal B' )  \propto \cal E(O_z(1/n))$, and $ (\cal G_{K_l} (z)| \cal B' )  \propto_{\vertif{\cdot} }\cal E(O_z(1/\sqrt n)) $.
         \item $( g_{K_l} (z) | \cal B' ) \in g_{\chi_n^{(l)}}(z)\pm \cal E\pt{ O_z(\hat \epsilon_n^{(l)}) }$ and  $(\cal G_{K_l} (z)| \cal B' )\in_{\vertiii{\cdot}}  \bf G_l(z) \pm \cal E\pt{ O_z(\hat \epsilon_n'^{(l)}) }  $.
     \end{enumerate}
\end{thm}

\begin{proof} By induction on the layers using Proposition \ref{ConcProd}, we can prove that there is an event $\cal B' \subset \cal B$ with $\Pp(\cal B'^c) \leq \Pp(\cal B^c)+ce^{-n/c}\leq O(\sqrt {\log n}/n)$, such that $(X_l | \cal B' )\propto_{\vertif{\cdot} }\cal E(1)$ for all layers $l \in \lint 1,L \rint$. The concentration properties of $\cal G_l(z)$ and $g_l(z)$ follow. Again by induction, 
using Lemma \ref{PropOrth} on $\E[X_{l+1}|X_l]$ proves that on $\cal B'$, $\vertii{\vec{\rm{diag}}\pt{\Delta_{X_l}}}$ and $\vertiii{K_l}$ remain bounded, and that $\vertim{\Delta_{X_l}} \leq O(\epsilon_n+\sqrt{\log n / n})  \leq O_z(\epsilon_n) $  for all $l \in \lint 1, L\rint$.  The random matrices $(X_l|\cal B')$ satisfy the Assumptions \ref{ANNRndAss} uniformly, and by repeatedly using Theorem \ref{ANNRndDetEq} we get the deterministic equivalents for all layers. The error terms $\hat \epsilon_n^{(l)}$ and $\hat \epsilon_n'^{(l)}$ are  given by the formulas above Theorem \ref{ANNRndDetEq}.
\end{proof}

\begin{rem}\label{remsimpleps}  The above deterministic equivalents are only meaningful if the error terms $\hat \epsilon_n'^{(l)}$ and $\hat \epsilon_n^{(l)}$ vanish when $n \to \infty$. Similarly to Remark \ref{ANNDetAssrem1}, let us mention a few cases where their expressions may be greatly simplified:
\begin{itemize}
\item If $\zeta_1(\tilde f_l)=0$ for some layer $l$, then for all subsequent layers $k \geq l$ the error terms $\hat \epsilon_n'^{(k)}$ and $\hat \epsilon_n^{(k)}$ do not depend on $\hat \epsilon_n^{(0)}$ and $\hat \epsilon_n'^{(0)}$ anymore.
    \item If $\epsilon_n=o(n^{-1/4})$, then $\hat \epsilon_n^{(l)}=O\pt{\hat \epsilon_n^{(0)}+\sqrt n \epsilon_n^2}$.
    \item If $\zeta_2(\tilde f_l)=0$ and $\epsilon_n=o(n^{-1/6})$ , then $\hat \epsilon_n^{(l)}=O\pt{\hat \epsilon_n^{(0)}+\sqrt n \epsilon_n^3}$.
     \item If $\zeta_2(\tilde f_l)=\zeta_3(\tilde f_l)=0$, then $\hat \epsilon_n^{(l)}=O\pt{\hat \epsilon_n^{(0)}+\sqrt{\log n}/n}$.
      \item If $\epsilon_n=o(n^{-1/2})$, then $\hat \epsilon_n'^{(l)}=O\pt{\hat \epsilon_n^{(0)}+\hat \epsilon_n'^{(0)}+ n \epsilon_n^2}$.
    \item If $\zeta_2(\tilde f_l)=0$ and $\epsilon_n=o(n^{-1/3})$ , then $\hat \epsilon_n'^{(l)}=O\pt{\hat \epsilon_n'^{(0)}+ n \epsilon_n^3}$.
     \item If $\zeta_2(\tilde f_l)=\zeta_3(\tilde f_l)=0$, then $\hat \epsilon_n'^{(l)}=O\pt{\hat \epsilon_n'^{(0)}+1/\sqrt n}$.
\end{itemize}

In the case $\epsilon_n = O(\sqrt{\log n/n})$ corresponding to a data matrix $X_0$ with i.i.d. columns (see Proposition \ref{iidsample}), and starting from  typical $\hat \epsilon_n^{(0)}=O_z(1/n)$ and $\hat \epsilon_n'^{(0)}O_z(1/\sqrt n)$ equivalents for the Stieltjes transform and the resolvent of $K_0$ respectively, Theorem \ref{ANNMulDetEq} gives an  $O_z(\log n/\sqrt n)$ equivalent for the Stieltjes transform. The error for the resolvents does not vanish in general because of the $O_z(\zeta_2(\tilde f)^2 \log n)$ error term. If $\zeta_2(\tilde f)=0$ however, we obtain a $O_z((\log n)^{3/2}/\sqrt n)$ approximation.
\end{rem}

\begin{cor}\label{CorMulAnn}
   Uniformly under Assumptions \ref{ANNMulAss}:\begin{enumerate}
       \item $\verti{g_{K_l}(z) - g_{\chi_n^{(l)}}(z)} \leq \sqrt {\log n} \, O_z(\hat \epsilon_n^{(l)}) $ a.s., and $  \vertim{{\cal G_{K_l}(z) - \bf G_l (z) }}  \leq \sqrt {\log n} \, O_z(\hat \epsilon'_n)$ a.s.
    \item 
 If $\tilde f_l$ is not linear, or if the measures $\chi_n^{(0)}$ are supported on the same compact of $(0,\infty)$,  there exists $\theta >0$ such that $D(\mu_{K_l}, \chi_n^{(l)}) \leq O({\hat \epsilon_n^{(l)}}{}^{\theta})$ a.s. 
\item If moreover $\chi_n^{(0)}$ converges weakly to a measure $\chi_\infty^{(0)}$, and if all ratios $\gamma_n^{(k)} \to \gamma_\infty ^{(k)}>0$ , then $\mu_{K_l}$ converges a.s. to the measure $\chi_\infty^{(l)}$ defined by induction as:
\[ \chi_\infty^{(l)} = \mmp\pt{\gamma_\infty^{(l)} } \boxtimes \pt{\frak a_l + \frak b_l \chi_\infty^{(l-1)}} . \] More precisely:      \[ D(\mu_{K_l}, \chi_\infty^{(l)} ) \leq O\pt{D(\chi_n^{(0)},\chi_\infty^{(0)})+\max_{0\leq k\leq l} |\gamma_n^{(k)}-\gamma_\infty^{(k)}|+{\hat \epsilon_n^{(l)}}{}^{\theta}}\quad \rm{a.s.}\]
   \end{enumerate}
\end{cor}

Again we  will not prove this result here, but refer to the proof of Corollary \ref{CorDetAnn} which is similar.
\begin{rem} Our result generalizes previously known global laws on the Conjugate Kernel model. Adapted to our notations, \cite[Theorem 3.4]{FW20} states that,  without bias in the model, if $f$ is twice differentiable, $\epsilon_n=o(n^{-1/4})$, and $\chi_n^{(0)}$ converges weakly to $\chi_\infty^{(0)}$, then $\mu_{K_l}$ converges weakly to $\chi_\infty^{(l)}$ a.s. Taking into account the Remark \ref{remsimpleps}, in this setting we have $\hat \epsilon_n ^{(l)} = O(\hat \epsilon_n ^{(0)}+\sqrt n \epsilon_n^2) = O(\hat \epsilon_n ^{(0)}+o(1)) $, and we retrieve the a.s. convergence of $\mu_{K_l}$ towards $\chi_\infty^{(l)}$ weakly, supplemented with quantitative estimates for the Stieltjes transforms and the Kolmogorov distances.

\end{rem}
\newpage
\section{Appendix: Bounds on Kolmogorov distances between empirical spectral measures} \label{Sec8}

Let us remind the notations $\cal F_\nu$ for the cumulative distribution function of a measure $\nu$, and $D(\nu,\mu) = \sup_{t \in \R} \verti{\cal F_\nu(t) - \cal F_\mu(t)}$ for the Kolmogorov distance between two measures $\nu$ and $\mu$. 

It is a well-known fact that the convergence in Kolmogorov distance implies the weak convergence for probability measures, and there is even an equivalence if the limiting measure admits a Hölder continuous cumulative distribution function (\cite{geronimo2003necessary}). In  \cite{banna2020holder} and \cite{moi1}, the authors propose a general method to derive a convergence speed in Kolmogorov distance from estimates on the Stieltjes transforms. This method implies for instance our Proposition \ref{OzandKol}. However the techniques employed are not well suited to work with two discrete measures like empirical spectral distributions.

Two matrices close in spectral norm admit the same limiting spectral distribution if it exists. This does not imply any bound on the Kolmogorov distances however in general because the measures are discrete. In this section, we show how a quantitative result may still be obtained, provided the limiting empirical measure is regular enough. We strongly incite the reader to first examine \cite[Section 8]{moi1} where the technical tools are explained in full details.

\begin{prop}\label{thmSigmaLKol}
    Let $\Sigma$ and $\Sigmat \in \R^{p \times p}$ be symmetric matrices such that:
    \begin{enumerate}
    \item $\vertiii{\Sigma}$ and $\vertiii{\Sigmat}$ are bounded, and $\vertiii{\Sigma - \Sigmat}$ converges to $0$.
        \item $\mu_\Sigmat$ converges weakly to some probability measure $\nu^\infty$, and $\cal F_{\nu^\infty}$ is Hölder continuous for some parameter $\beta > 0$.
    \end{enumerate}
    Then $\mu_\Sigma$ converges weakly to $\nu^{\infty}$, and more precisely in Kolmogorov distance:\[  D\pt{\mu_\Sigma,\nu^{\infty}} \leq O \pt{ \vertiii{\Sigma - \Sigmat}^{\frac \beta {4 + 2 \beta}}+ D\pt{\mu_\Sigmat,\nu^\infty}} .\]
\end{prop}

\begin{lem} \label{FixedBound} For any $y \in (0,1)$ and $A>0$: \begin{align*} D(\mu_\Sigmat,\mu_\Sigma)  \leq O \pt{ A \, \frac{ \vertiii{\Sigma - \Sigmat}}{y^2 }   + \frac{ 1} {y^2 A } +  y^\beta + D\pt{\mu_\Sigmat,\nu^\infty}} . \end{align*}
\end{lem}

\begin{proof} We closely follow the \cite[Section 3.1]{moi1} with the only key difference that we plug in Bai's Inequality the following bound:
\begin{align*} \verti{\cal F_{\mu_\Sigmat}(x+t)-\cal F_{\mu_\Sigmat}(x)}   &\leq  \verti{\cal F_{\nu^\infty}(x+t)-\cal F_{\nu^\infty}(x)}  + 2 \, D\pt{\mu_\Sigmat,\nu^\infty}.
    \end{align*}We thus obtain:
\begin{align*} D(\mu_\Sigmat,\mu_\Sigma) &\leq \frac 2 \pi \pt{ \int_{\R} \verti{g_{\mu_\Sigmat}-g_{\mu_\Sigma}}(t + i y) dt + \frac 1 y \sup_{x \in \R} \int_{\br{\pm 2 y \tan\pt{\frac {3 \pi} 8 }}} \verti{\cal F_{\nu^\infty}(x+t)-\cal F_{\nu^\infty}(x)} dt }\\
&\quad + O\pt{D\pt{\mu_\Sigmat,\nu^\infty}}. \end{align*}
 For $z \in \C^+$ a classical application of the resolvent identity gives $\verti{g_\Sigma(z)-g_\Sigmat(z)} \leq \frac {\vertiii{\Sigma - \Sigmat}}{\Im(z)^2}$. The rest of the proof is exactly the same as in \cite[Section 3.1]{moi1}.
\end{proof}

\begin{proof}[Proof of Theorem \ref{thmSigmaLKol}]
    We optimize $y$ and $A$ in the above lemma by choosing $y_n = \vertiii{\Sigma - \Sigmat}^{\frac 1 {4 + 2 \beta}}$ and $A_n ={\vertiii{\Sigma - \Sigmat}^{-1/2}}$, which leads to the bound $D(\mu_\Sigmat,\mu_\Sigma)  \leq O \pt{ \vertiii{\Sigma - \Sigmat}^{\frac \beta {4 + 2 \beta}}+ D\pt{\mu_\Sigmat,\nu^\infty}}$. A final triangular inequality proves the proposition.
\end{proof}

\newpage

\printbibliography

\end{document}